\documentclass[11pt,authoryear,table,nopreprintline]{elsarticle}
\usepackage{setspace}
\usepackage{comment}
\usepackage{chngcntr}
\usepackage{subfiles}
\usepackage{amsmath}
\usepackage{amssymb}
\usepackage{amsthm}

\DeclareMathOperator*{\argmin}{arg\,min}
\usepackage{array}
\usepackage{multirow}
\usepackage{multicol}
\usepackage{booktabs}
\usepackage{wrapfig}
\usepackage{caption}
\usepackage{subcaption}
\usepackage{float}
\usepackage{longtable}
\usepackage{pgfplots}
\pgfplotsset{compat=1.17}

\usepackage{algorithm}
\usepackage{algpseudocode}
\makeatletter
\renewcommand{\ALG@name}{Algorithm}
\makeatother

\algrenewcommand\algorithmicrequire{\textbf{Input:}}
\algrenewcommand\algorithmicensure{\textbf{Output:}}
\algrenewcommand{\algorithmiccomment}[1]{\hfill{\small\textcolor{gray}{// #1}}}
\algrenewcommand\textproc{\textit} 
\algrenewcommand{\alglinenumber}[1]{\scriptsize #1:} 

\usepackage{csquotes}
\usepackage{siunitx}
\usepackage{hyperref}
\usepackage[margin=1in]{geometry}

\usepackage{microtype}
\usepackage{enumitem}
\setlist{noitemsep, topsep=2pt, partopsep=0pt, parsep=2pt}
\setlist[description]{noitemsep, leftmargin=0pt, labelindent=0pt, topsep=2pt}

\AtBeginDocument{
  \setlength{\abovedisplayskip}{4pt}
  \setlength{\belowdisplayskip}{4pt}
  \setlength{\abovedisplayshortskip}{2pt}
  \setlength{\belowdisplayshortskip}{2pt}
}

\allowdisplaybreaks

\makeatletter
\renewcommand{\section}{\@startsection{section}{1}{\z@}%
  {-2.0ex plus -0.5ex minus -.2ex}%
  {1.0ex plus .2ex}%
  {\normalfont\large\bfseries}}
\renewcommand{\subsection}{\@startsection{subsection}{2}{\z@}%
  {-1.6ex plus -0.5ex minus -.2ex}%
  {0.7ex plus .2ex}%
  {\normalfont\normalsize\bfseries}}
\renewcommand{\subsubsection}{\@startsection{subsubsection}{3}{\z@}%
  {-1.2ex plus -0.5ex minus -.2ex}%
  {0.5ex plus .2ex}%
  {\normalfont\normalsize\bfseries\itshape}}
\makeatother

\newtheorem{claim}{Claim}

\definecolor{safeblue}{RGB}{0,90,181}

\definecolor{safered}{RGB}{220,50,32}

\definecolor{purple}{RGB}{128,0,128}

\newcommand{\scalefigure}[1]{%
  \IfFileExists{#1}%
    {\includegraphics[width=0.8\textwidth]{#1}}%
    {\fbox{\parbox{0.78\textwidth}{\centering\vspace{1em}\textit{[Missing figure: \detokenize{#1}]}\vspace{2em}}}}%
}

\begin{document}

\begin{frontmatter}

\title{A Policy Decomposition Framework\\ for Dynamic Order Fulfillment Operations}

\author[label1]{Gal Neria\corref{cor1}}
\ead{galneria@mit.edu}

\author[label2]{Michal Tzur}
\ead{tzurm@tauex.tau.ac.il}

\author[label3]{Marlin W. Ulmer}
\ead{marlin.ulmer@ovgu.de}

\cortext[cor1]{Corresponding author}

\affiliation[label1]{organization={Massachusetts Institute of Technology},
            city={Cambridge},
            postcode={02139},
            state={MA},
            country={USA}}

\affiliation[label2]{organization={School of Industrial and Intelligent Systems Engineering, Tel Aviv University},
            city={Tel Aviv},
            postcode={6997801},
            country={Israel}}

\affiliation[label3]{organization={Otto-von-Guericke Universitat Magdeburg},
            city={Magdeburg},
            country={Germany}}

% Note: we are only allowed 250 words in the abstract!
\begin{abstract}
Modern supply chains span diverse operational environments, ranging from e-commerce distribution networks to customized production-to-order manufacturing lines. Across these settings, operational efficiency depends on tightly coordinating two highly interdependent stages: order preparation and downstream delivery. Although these stages are traditionally managed in isolation, real-world fulfillment systems must satisfy stringent delivery expectations under dynamic stochastic order arrivals. To bridge this gap, we introduce the Dynamic Order Fulfillment Problem (DOFP), a new problem class unifying logistical challenges previously studied separately. We model DOFP as a Markov decision process whose state and decision spaces are partitioned into preparation and delivery sub-spaces, linked by synchronization constraints. While state-of-the-art approaches attempt to optimize both fulfillment stages simultaneously over myopic rolling horizons, our framework isolates and optimizes the downstream delivery policy using stochastic look-ahead anticipation, treating preparation strictly as a state-level constraint filter. To solve this, we develop the Decomposition-Driven Framework with Value Function Approximation (DDF-VFA), which utilizes a novel policy-level decomposition. This design partitions the search into a delivery-stage master problem and a preparation-stage compatibility subproblem, iteratively refined via feedback loops. DDF-VFA executes this strategy by combining a large-neighborhood search over partial delivery decisions with a neural-network value function approximation for the cost-to-go. Numerical illustrations on two example variants using real-world datasets show that DDF-VFA consistently outperforms benchmarks that optimize the two stages independently or jointly without decomposition. Finally, the framework naturally scales to accommodate additional real-world complexities as modular add-ons, including batched or multi-stage preparation, and heterogeneous resources.
\end{abstract}


\end{frontmatter}
%%%%%%%%%%%%%%%%%%%%%%%%%%%%%%%%%%%%%%%%%%%%%%%%%%%%%%%%%%%%%%%%%%%%%%
% Samples of sectioning (and labeling) in TRSC
% NOTE: (1) \section and \subsection do NOT end with a period
%       (2) \subsubsection and lower need end punctuation
%       (3) capitalization is as shown (title style).
%
%\section{Introduction.}\label{intro} %%1.
%\subsection{Duality and the Classical EOQ Problem.}\label{class-EOQ} %% 1.1.
%\subsection{Outline.}\label{outline1} %% 1.2.
%\subsubsection{Cyclic Schedules for the General \textbf{deterministic} SMDP.}
%  \label{cyclic-schedules} %% 1.2.1
%\section{Problem Description.}\label{problemdescription} %% 2.
% Text of your paper here
%\DoubleSpacedXI

\section{Introduction} \label{sec: intro}
%please so not remove this: {clear structure: problem field, problems, models, method, results and contribution. preparation phase, and delivery phase...}
\noindent Driven by the rapid expansion of e-commerce and on-demand delivery services, {global logistics spending reached an estimated \$12.8 trillion in 2024 and is projected to exceed \$14 trillion by 2029 \citep{statista2024}}. These industries, encompassing warehousing, production, and food services, face increasing pressure to deliver orders promptly and efficiently \citep{wassmuth2023demand,chen2024courier}. This pressure often leads to overlapping capture and fulfillment phases, where orders arrive while deliveries are still ongoing \citep{fleckenstein2023recent, ErazoToriello2025}. Addressing these challenges requires seamless real-time coordination across preparation and delivery stages.

% For example
%\paragraph{Real-World Applications.}  
%The synchronization of preparation and delivery stages is crucial across various industries. 
In e-commerce, companies such as Walmart must align warehouse picking operations with delivery schedules to ensure timely and cost-effective fulfillment of shipments. Competitive pressures drive continual reductions in delivery times, as seen in Amazon’s response to local Berlin competitors by introducing 60-minute delivery deadlines \citep{ulmer2017delivery}.
% In e-commerce, companies such as Walmart must align warehouse operations with delivery schedules to ensure timely, cost-effective shipments. 
Similarly, on-demand production requires precise coordination between manufacturing and distribution. This is evident in firms that must deliver customized materials to construction sites just in time \citep{dallasega2017method}. In manufacturing sectors such as metalworking and packaging, complex production planning with setup times must align with outbound distribution to maximize resource utilization \citep{berghman2023review}. 
In meal delivery platforms, synchronized meal preparation and driver dispatch are required, with added complexity when combining orders from multiple restaurants, as early preparation risks freshness, and delays affect customer satisfaction.
Lastly, in healthcare and humanitarian logistics, coordination between preparation and delivery is crucial to ensure the efficient distribution of pharmaceuticals, chemotherapy treatments, radio-pharmaceuticals, food donations, and blood supplies, avoiding waste and ensuring critical medicines reach patients on time \citep{eisenhandler2022multi, torrado2022towards,li2024model}.

%problem structure
%\paragraph{Problem Structure and General Challenges.}
Across the aforementioned diverse industries, these fulfillment applications share a core structure of two interdependent stages: preparation and delivery. Both stages require complex real-time optimization in a combinatorial environment, with preparation constrained by limited resources and delivery incurring most or all real-time decision-dependent costs, such as routing expenses and customer delays. {This cost asymmetry is a key insight supported by prior literature: preparation costs are often relatively fixed, largely determined by the number of available pickers, personnel, or machines, and are not highly sensitive to small variations in task timing. In contrast, costs in delivery operations such as fuel consumption or non-monetary cost of unpunctuality are highly variable, depending on dynamic routing decisions and dispatch timing.} The dynamic nature of order arrivals demands flexible, adaptive decision-making to ensure synchronized and efficient operations.

%\paragraph{Modeling.} 
To address such settings, we introduce the Dynamic Order Fulfillment Problem (DOFP), a new problem class that unifies diverse domains in which 
synchronization of preparation and delivery stages under dynamic, stochastic conditions is desired.
We model the class using a novel Markov decision process (MDP), which partitions the state space into two sub-spaces, each associated with a decision sub-space, representing the preparation and delivery stages. Uniquely, the decision sub-spaces are interdependent as they are linked by a set of constraints that represent the required synchronization in two-stage systems. 

Derived from our unique MDP formulation, we propose a decomposition framework for policy optimization that draws inspiration from the iterative refinement logic of traditional decomposition frameworks, but is fundamentally different in structure and scope. While classical decomposition methods operate over static decision variables in deterministic or two-stage stochastic settings, our method decomposes over a space of policies in a real-time, dynamic, and stochastic environment. Specifically, we separate the policy search into a master problem that proposes delivery policies and a subproblem that verifies the existence of compatible preparation policies. This formulation extends decomposition principles to the policy level within an MDP framework, enabling tractable optimization under uncertainty and temporal interdependence.

Building on this decomposition for the DOFP class, we propose the Decomposition-Driven Framework with Value Function Approximation (VFA) (\textit{DDF-VFA}). {This solution method leverages the policy decomposition to alternate between solving the master and subproblem iteratively, using feedback mechanisms to refine decisions across stages.}
%This solution framework is rooted in the policy decomposition, which partitions the problem into a master problem and a sub-problem, solved iteratively. 
That is, the \textit{DDF-VFA} includes methods to solve the policy master problem and the policy sub-problem and iterate between them using feedback loops. In this way, our solution framework utilizes the unique problem structure of the class of problems that we address.

In our implementation, the policy master problem is optimized using a large neighborhood search (LNS) heuristic which searches over partial solutions of the second-stage decisions while an innovative auxiliary algorithm completes them. To address the subproblem, the Synchronized Preparation Scheduling (SPS) algorithm rapidly constructs a compatible preparation-stage decision or generates feasibility cuts. Finally, the evaluation employs a VFA to estimate the long-term impacts of decisions, ensuring a balanced approach to immediate and future objectives.

In a meta-analysis, we specify and implement the \textit{DDF-VFA} for two selected problems from the two major areas of warehousing and production, each with its own domain-specific constraints and objectives. We compare our solution framework to several benchmark policies from the literature. We show that our approach provides the best performance and is robust to changes in data, constraints, and objective functions. For each component of the \textit{DDF-VFA}, we further present an alternative to illustrate their value and functionality.  

\subsection{Summary of Contributions}
On the theoretical side, we define the DOFP class (Section~\ref{sec: problem statement}), present a novel MDP formulation (Section~\ref{sec: MDP}), and introduce a policy-based decomposition approach to MDPs (Section~\ref{sec: decomposition}). Methodologically, we propose the \textit{DDF-VFA} solution framework tailored to the DOFP class structure (Section~\ref{sec: solution method description}), enhance it with specialized feature add-ons for broader applicability, and demonstrate its adaptability on two representative problems (Section~\ref{sec:model_problems}). Computationally, we demonstrate superior performance over previous approaches on two problems and two real-world datasets (Section~\ref{sec: results}), and present a component-level analysis showing robustness across varying objectives, constraints, and data conditions. %To further establish the structural generalizability of our framework, we also evaluate its performance on a pharmacy prescription delivery case study, detailed in \ref{app: pharmacy variant}. 
In the next section, we position our contributions in the related literature.

\section{Literature Review}\label{sec:literature_review}
\noindent {We organize the related literature in two streams. Section~\ref{sec:lit application} reviews application-domain work on integrated preparation-delivery operations. Section~\ref{sec:lit decomposition} discusses methodological work on decomposition and reinforcement-learning approaches that inform our policy decomposition.}

\subsection{Integrated Preparation-Delivery Operations}\label{sec:lit application}

The integration of preparation and delivery operations has attracted growing attention across multiple domains. Recent approaches address this systemic challenge through joint optimization models that attempt to synchronize both fulfillment stages concurrently. However, while making progress relative to disjoint optimization, these approaches still fail to fully exploit fundamental structural characteristics typical to preparation-delivery operations. Specifically, while variable costs concentrate heavily in the delivery stage and preparation capacity is largely determined by fixed staffing decisions made independently of real-time operations, existing frameworks treat both stages equally. Consequently, they fail to leverage the cost asymmetry inherent in these systems, which suggests prioritizing downstream delivery efficiency subject to upstream preparation feasibility. The following review examines how this tension manifests across literature streams.

In integrated picking and delivery, while several studies in static and deterministic settings demonstrate the clear benefits of joint planning, they typically treat the optimization of both stages uniformly rather than tailoring the solution methods to their distinct characteristics.
%reflect this inherent modeling symmetry.
\citet{neves2019solving} consider fixed wages in the preparation stage, meaning that operational costs accrue primarily in delivery; \citet{jamili2022quantifying} examine collaborative warehouses with shared resources and delivery-stage tardiness; and \citet{rijal2023dynamics} show that joint picking-delivery optimization outperforms sequential decision-making. Only a few studies explicitly consider uncertainty: \citet{raj2024stochastic} address tactical-level decisions through simulation, and \citet{d2024integrating} propose a metaheuristic for dynamic operations that minimizes tardiness and travel time without anticipatory mechanisms. Following \citet{rijal2023dynamics} and \citet{d2024integrating}, we adopt the \textit{Integrated} method (defined in Section~\ref{sec: benchmarks}) as a baseline benchmark.

In integrated production and delivery, studies similarly emphasize the dominance of outbound logistics expenses over preparation constraints. \citet{ganji2022new} show that aligning production and delivery schedules with respect to due dates and time windows reduces costs, emissions, and customer dissatisfaction; further objectives include delivery tardiness \citep{ullrich2013integrated}, time-window violations \citep{wu2022variable}, and total delivery times \citep{pereira2022hybrid}. Dynamic and stochastic settings remain less explored: \citet{liu2022approximate} integrate production and delivery under dynamic arrivals with fixed vehicle departure times, and \citet{KERMANI2024} use a two-stage stochastic program with uncertain demand; for a recent review, see \citet{berghman2023review}. While these works heavily weigh delivery-side penalties in their objective functions, they treat both stages similarly during optimization rather than decoupling them to exploit the rigid nature of preparation capacities relative to variable delivery costs.

Related streams of literature also arise in the context of meal preparation and delivery, where researchers primarily focus on optimizing courier logistics \citep{ulmer2021restaurant}. These studies implicitly isolate delivery costs by treating meal preparation durations as exogenous constants or simple stochastic invariants, thereby focusing exclusively on minimizing courier delays or maximizing fulfilled orders within strict deadlines. However, by omitting the active scheduling of the preparation stage, they do not address the operational interdependencies between processing and distribution, limiting their ability to optimize both stages simultaneously. An exception is \citet{NeriaGhost2024}, who study ghost kitchens sharing a facility. They propose an integrated anticipatory method for combined order preparation and delivery, utilizing an LNS to explore decisions that are subsequently evaluated by a VFA; a technique we adapt as an additional benchmark (referred to as \textit{AI}).

Despite these advances, the literature lacks a unified framework for integrated preparation-delivery operations under dynamic and stochastic order arrivals with explicit anticipation of future arrivals across both stages. Moreover, existing solution methodologies consistently fail to exploit the concentration of variable operational costs within the downstream delivery stage. By assigning equal attention to both decision spaces, conventional joint optimization introduces prohibitive computational burdens. Our policy-level decomposition framework directly mitigates this algorithmic challenge by isolating preparation as a feasibility-only subproblem.

\subsection{Decomposition and Reinforcement-Learning Approaches}\label{sec:lit decomposition}
Decomposition methods address the computational challenges of large-scale optimization by dividing problems into smaller, more manageable components that can be solved iteratively \citep{wolsey2014integer}. They have proven highly effective in combinatorial optimization \citep{fontaine2023branch,godbersen2024robust}.
A notable approach is the Dantzig-Wolfe decomposition that partitions the decision space into subproblems coordinated by a master problem \citep{Dantzig1960}. Extensions are successfully applied to non-convex problems \citep{byers2023long, jacquillat2024branch}, multi-stage stochastic programming \citep{singh2009dantzig}, and other problems with stochastic elements \citep{rostami2021branch}. These advancements demonstrate the flexibility of the Dantzig-Wolfe framework in addressing uncertainty and complex constraints. However, they remain focused on static decision-making.

Another widely used approach, the Benders decomposition, iteratively refines solutions by partitioning decision variables into a master problem and subproblem(s) \citep{Benders1962}. Although highly effective for problems with separable structures \citep{Geoffrion1972}, it struggles with inseparability and synchronization requirements. Several extensions to two-stage stochastic programming \citep{wang2020stochastic}, and Bilevel optimization \citep{beck2023survey}
are presented.
Recent advancements introduce hybrid methods \citep{rostami2023convex}. While distinct from Benders decomposition, these methods share its iterative refinement notion. {A related contribution is the work of \citet{liu2023demand}, who apply Benders decomposition to a deterministic MILP approximation of a dynamic last-mile delivery problem. However, their method does not decompose a policy space or operate on the full state dynamics of an MDP, but rather optimizes on a tractable surrogate model.} Applying such approaches to dynamic systems with evolving states and interdependent decisions remains unexplored.

Reinforcement Learning (RL) decompositions offer another perspective, classified broadly into Hierarchical RL (HRL) and Multi-Agent Systems (MAS). HRL simplifies the state space but often assumes a small decision space, limiting its focus to state complexity rather than decision complexity \citep{barto2003recent,vezhnevets2017feudal,zhang2026enhancing}. Multi-Agent Systems (MAS) distribute decisions across agents to reduce computational complexity but typically assume independence between agents, which restricts their applicability to tightly coupled problems \citep{lowe2017multi,zhang2021multi,ahadi2023cooperative, guillet2023coordinated}.

An exception is the work of \citet{kullman2021electric}, which addresses an electric vehicle routing problem with known customer locations and uncertainty at charging stations. Their Benders-based decomposition method exploits the problem’s structure by separating deterministic routing decisions from recharging decisions, which are solved in a subproblem using queuing theory approximations. Our work extends decomposition techniques to a broader class of dynamic stochastic systems, explicitly addressing synchronization constraints and iterative feedback between decision stages. Furthermore, our framework uniquely integrates immediate and future costs through dynamic feedback, enabling coordinated decision-making across interdependent stages.

\section{The DOFP Class Definition}\label{sec: problem statement}

In this section, we introduce the DOFP class of problems, which models the complexities of a two-stage fulfillment operation. Note that the term \enquote{two-stage} refers to the sequential nature of fulfillment operations (comprising preparation followed by delivery) and does not imply a two-stage stochastic programming formulation. Our model captures real-time decision making under dynamic and stochastic order arrivals.
We begin with a general description in Section~\ref{sec:Class Description}, detailing its core structure and components. In Section~\ref{sec: example}, we provide an example to illustrate a problem in this class. Next, Section~\ref{sec: MDP} presents an MDP formulation for any problem in this class, capturing its inherent dynamics and uncertainties.

\subsection{Class Description}\label{sec:Class Description}
To precisely characterize the types of problems that belong to the DOFP class, we outline the key defining elements below.

\begin{enumerate}
    \item Orders arrive dynamically and stochastically over time at a facility in a given location. 
    \item Each order has to be prepared and then delivered to a destination location indicated by the order.
    \item There are two sets of disjoint resources: preparation and delivery. Typically, the second stage resources are vehicles. The resources in each set could be either homogeneous or heterogeneous and are responsible for performing the preparation and delivery operations subject to various constraints.
    The preparation stage may represent complex environments, e.g., {sequence dependent set-up times between orders,} multiple sub-stages or simultaneous processing of multiple orders (batching), as long as any costs associated with this stage remain fixed, independent of the specific operational decisions made within this stage. The delivery stage may represent an intricate process, e.g., delivery to a micro-hub from which delivery is made to the order's destination. 
    \item Upon the arrival of the $k$th order, there are two types of decisions to be made: $x_k^{(1)}$ that allocates preparation resources to the new order and possibly modifies previous allocations of preparation resources to other orders; 
    and $x_k^{(2)}$ that allocates delivery resources to the new order, and possibly modifies previous allocations of delivery resources to other orders.
    \item Synchronization between the preparation and delivery stages is essential to maintain feasible and efficient operations. Specifically, timing must be coordinated so that prepared orders are ready when delivery resources dispatch them. The extent of required coordination may vary across different problem variants within this class.
    \item The delivery operations incur variable costs that depend on the decisions, whereas the costs of preparation operations are fixed and do not depend on operational decisions. {This cost asymmetry reflects industry practice in which preparation capacity is determined by the number of available pickers, personnel, or machines on shift and is largely insensitive to small variations in task timing, while delivery costs (routing, fuel, unpunctuality) vary directly with operational decisions \citep{moons2018integration, neves2019solving, rijal2023dynamics, NeriaGhost2024}.}
    \item Decisions made in both stages are non-preemptable: in-progress preparation tasks must be completed, and a vehicle already en route must finish its current trip before its plan can be revised at a subsequent decision epoch. This design mirrors dispatching and production practices in dynamic fulfillment systems, avoiding the significant coordination overhead required to interrupt tasks or recall vehicles mid-route.
\end{enumerate}
For convenience, throughout the remainder of the paper, we refer to preparation resources simply as "resources" and to delivery resources as "vehicles".

\subsection{Example}\label{sec: example}
To illustrate the components of the DOFP class and to prepare the model, we present a small example of a state of the problem in Figure \ref{fig:example}. To keep the presentation simple, we omit detailed mathematical notation. 
\begin{figure}[h!]
   \centering
   \caption{Example for a general state of a problem in the DOFP class}  
\includegraphics[width=0.81\textwidth]{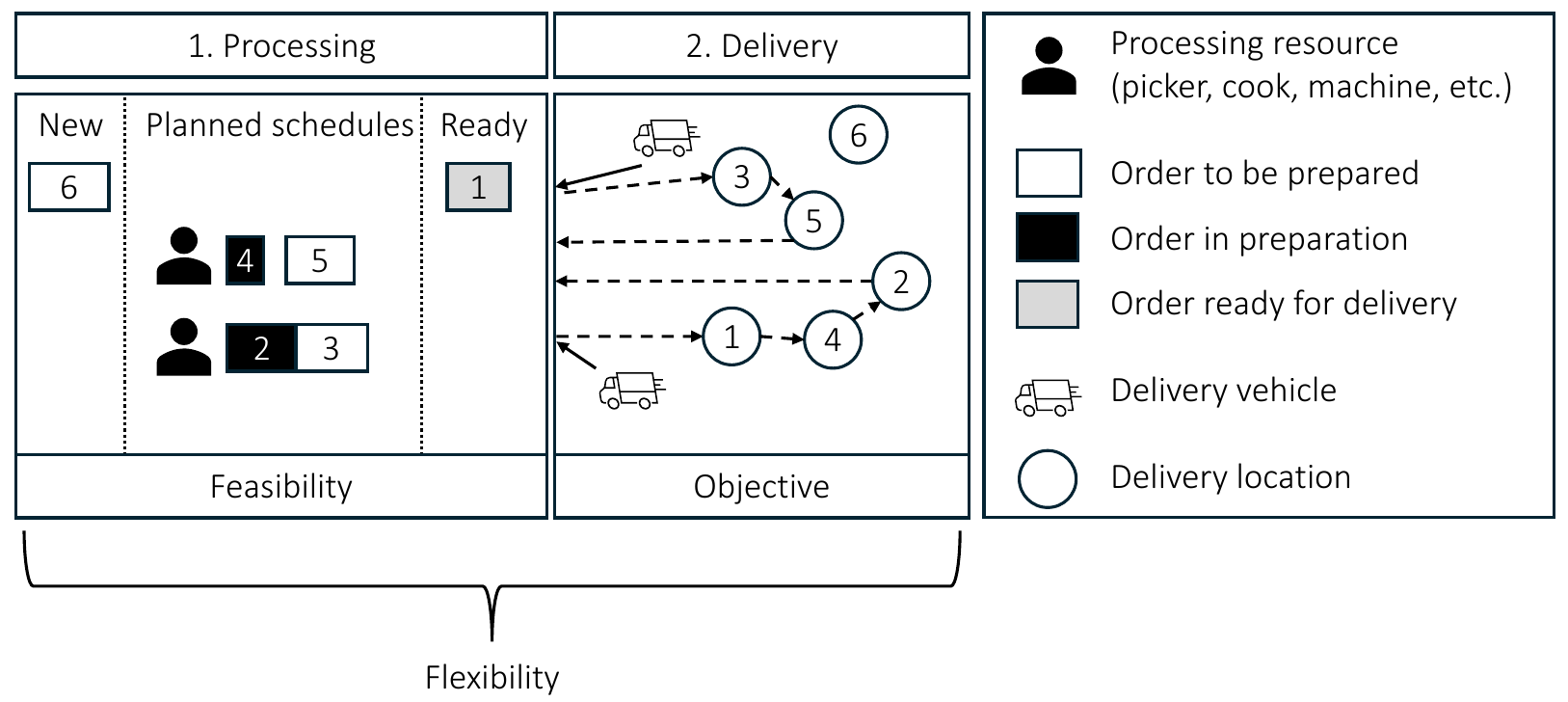}
%\scalefigure{figures/example.pdf}
   \label{fig:example}
\end{figure}

The figure summarizes an example state of the system. In this state, six orders are managed, represented as rectangles. White rectangles represent orders for which preparation has not started yet. The width of each white rectangle corresponds to the order’s processing time, while black rectangles indicate orders currently in preparation, with their remaining processing time shown by their widths. Two resources (machines or employees) handle these orders. Order 1 has already been processed and is ready for delivery, indicated by its grey shade. Orders 4 and 2 are in preparation, while orders 5 and 3 are scheduled next. The small gap between orders 4 and 5 might be the result of some additional constraint such as setup times, perishability of orders, or time windows. Order 6 has just arrived and is not yet assigned to any resource. The delivery fleet consists of two vehicles, each currently on their way. Their return times are reflected by the length of the solid line. One vehicle is scheduled to deliver orders 3 and 5 once it returns and both are ready. The second vehicle is scheduled to deliver orders 1, 4, and 2 once it returns and they are ready. Each trip has a planned departure time, ensuring all assigned orders are prepared, and the vehicle is available at departure. While in the example, only one planned trip per vehicle is scheduled, in the model, multiple subsequent planned trips per vehicle are possible in a state.

A decision in this state has several components. First, with respect to the processing stage, a decision must be made about how to integrate the new order 6 into the existing schedules. One option is to add order 6 at the end of the current preparation sequence on one of the resources in a first-in, first-out manner. In that case, if no constraints are violated, order 6 might be scheduled after order 5. Alternatively, order 6 could be inserted before order 5 or 3, changing the sequence. It is also possible to reallocate orders between resources, such as assigning order 6 to the top resource while moving order 5 to the bottom resource before order 3. Besides sequencing, decisions about the precise starting times for each order’s preparation are required. All components of a decision regarding the first stage need to ensure feasibility on the first stage and in connection to the second stage. In the second stage, the delivery, the components of the decision contain the assignment of the new order to a trip and the update of the planned trips. For example, order 6 could be added to one of the scheduled trips or assigned a subsequent delivery trip. Alternatively, orders might be reassigned and new trips generated, e.g., by creating three trips for orders 1 and 4, for orders 2 and 5, and for orders 3 and 6. Again, all trips require determination of timing and need to ensure feasibility with respect to the problem constraints. In the second stage the objective is realized and affects the decisions.

Even this small example already highlights the complexity of the decision space with the challenge of finding an effective second stage plan with respect to the objective while ensuring feasibility in connection to the first stage and in the first stage itself. At the same time, since the problem is stochastic and dynamic, all resources should be kept flexible with respect to future developments.

\subsection{MDP Formulation of the DOFP}\label{sec: MDP}

We next present an MDP formulation for a general problem in the DOFP class, following the modeling framework of \citet{ulmer2020modeling} {(route-plan based model)}. The DOFP involves two tightly interdependent stages: preparation and delivery. To effectively capture this structure, the MDP explicitly represents these stages by separating the state and decision variables. This decomposition emphasizes the synchronization challenges between stages, enhances interpretability, and supports the decomposition-based solution frameworks introduced later. Below, we define the key components of the MDP formulation, starting with general notation.

Orders arrive during $[0, T^c]$ and are processed and delivered during $[0, T]$ ($T>T^c$). 
We represent the complete set of orders by $I$. Each order $i\in I$ is associated with an arrival time to the system \( t^o_i \) and a customer location $l_i$ (the order's destination), as well as a set of problem-specific attributes. The set of resources (of the first stage) is denoted by $C$. The set \( V \) represents the vehicles. 
We define the facility's location, where orders are prepared, as 0. The travel time between the locations of orders \(i \in I \cup \{0\}\) and \(j \in I \cup \{0\}\) (with \(i = 0\) or \(j = 0\) representing the facility) is denoted by \(t^t_{ij}\).

\paragraph{Decision Points}
A decision point \(k\) arises each time a new order is placed, as well as at time \(T^c\). Given that the arrival of orders is stochastic, the total number of decision points \(K\) is a random variable.

\paragraph{States}
A state \( S_k \in \mathcal{S} \) at each decision point \( k \) is defined as \( S_k = (S_k^{(1)}, S_k^{(2)}) \), where \( S_k^{(1)} \) and \( S_k^{(2)} \) represent the information relevant to the preparation and delivery stages, respectively. These sets may share mutual information critical to each of the stages and their coordination.

The sets \( S_k^{(1)} \) and \( S_k^{(2)} \) each contain information on the set of \emph{open orders} \( I_k \). This set includes all orders, along with the information that comes with them, that have been placed but not yet dispatched, including orders currently in preparation or completed but awaiting dispatch, as well as the newly arrived order \( i_k \). Specifically, we denote the information on open orders in each stage as \( I_k^{(1)} \) for the preparation stage and \( I_k^{(2)} \) for the delivery stage.
\( I_k^{(1)} \) includes details relevant to preparation, such as order-specific processing requirements, while \( I_k^{(2)} \) covers delivery-relevant information, including order-specific delivery requirements and customer locations.
Together, \( I_k = I_k^{(1)} \cup I_k^{(2)} \).

The set \( S_k^{(1)} \) contains $ I_k^{(1)}$, the details of the current first-stage plan denoted by \( \Psi_k \), and the associated starting times \( t^s_k \) for each preparation task (of one or several orders simultaneously).
The set \( S_k^{(2)} \) contains \( I_k^{(2)} \), the second-stage plan
\( \Theta_k \) which includes plans for all vehicles \( v \in V \), and information regarding times associated with planned trips and vehicles. Specifically, each vehicle's plan \( \Theta_{kv} \in \Theta_k \)  includes trips that are organized as a sequence \( \theta_{kv1}, \theta_{kv2}, \ldots \), with each trip containing specific order sequences and planned departure times \( t^d_k \) (either from a facility or from a customer, when applicable). The return times \( t^r_k \) for all vehicles indicate the next expected returns to the facility, accounting for whether each vehicle is idle or on a previous trip.
The end time of preparation for each order that began before \( t_k \), denoted \( t^{m}_k \), is derived from \( (\Psi_k, t^s_k) \in S_k^{(1)} \) and is included in \( S_k^{(2)} \) due to the required synchronization between the two stages. To summarize:
\begin{equation}\label{eq: state decomposition}
\resizebox{\textwidth}{!}{$
S_k^{(1)} = \big(\underbrace{t_k}_{\scriptscriptstyle\text{Time}}, \underbrace{I^{(1)}_k}_{\scriptscriptstyle\text{Orders}}, \underbrace{\Psi_k, t^s_k}_{\scriptscriptstyle\text{Resources}}\big),\;
S_k^{(2)} = \big(\underbrace{I^{(2)}_k}_{\scriptscriptstyle\text{Orders}}, \underbrace{t^{m}_k}_{\scriptscriptstyle\text{Sync}}, \underbrace{\Theta_k, t^d_k, t^r_k}_{\scriptscriptstyle\text{Vehicles}}\big),\;
S_k = \big(\underbrace{t_k}_{\scriptscriptstyle\text{Time}}, \underbrace{I_k}_{\scriptscriptstyle\text{Orders}}, \underbrace{\Psi_k, t^s_k}_{\scriptscriptstyle\text{Resources}}, \underbrace{\Theta_k, t^d_k, t^r_k}_{\scriptscriptstyle\text{Vehicles}}\big)
$}
\end{equation}
where \( t^{m}_k \) does not appear explicitly in $S_k$ because, as mentioned above, it is derived from \( (\Psi_k, t^s_k) \).

\paragraph{Decisions}
A decision \( x_k = (x_k^{(1)}, x_k^{(2)}) \in \mathcal{X}(S_k) \) at each decision point \( k \) consists of two components. The preparation plan update \( x_k^{(1)} = (\Psi_k^x, t_k^{s,x}) \in \mathcal{X}^{(1)}(S_k^{(1)}) \) incorporates the allocation of \( i_k \) to preparation resources and may involve rescheduling other orders or even modifying the entire plan, except for already started tasks; it satisfies a set of preparation constraints \( \mathcal{Q}_{\text{prep}}(S_k^{(1)}) \). The delivery plan update \( x_k^{(2)} = (\Theta_k^x, t^{d,x}_k) \in \mathcal{X}^{(2)}(S_k^{(2)}) \) incorporates \( i_k \) in the vehicles' schedule, possibly reassigning other orders or trips and adjusting them (and potentially modifying the entire vehicles' schedule), and satisfies a set of delivery constraints \( \mathcal{Q}_{\text{deliv}}(S_k^{(2)}) \).
A combined decision \( x_k = (x_k^{(1)}, x_k^{(2)}) \in \mathcal{X}(S_k) \) additionally meets a set of synchronization constraints \( \mathcal{Q}_{\text{sync}}(S_k) \); when this holds, we say that \( x_k^{(1)}\) and \( x_k^{(2)}\) are \emph{compatible decisions}.
At the final decision point \( K \) (time \( T^c \)), the decision \( x_K \) focuses solely on optimizing existing schedules and resource allocations without considering either a new order or future orders.

\paragraph{Costs}
In accordance with the definition of the DOFP, variable operational costs are incurred solely from the decisions $x_k^{(2)}$ and not $x_k^{(1)}$. Thus, the immediate cost associated with state $S_k = (S_k^{(1)}, S_k^{(2)})$ and decision $x_k^{(2)}$ is defined by
$\mathbb{C}(S_k^{(2)}, x_k^{(2)})$.

\paragraph{Post-Decision States}
A post-decision state $S_k^x \in \mathcal{S}^x$ captures the system's state immediately after a decision is made but before any new order arrives (i.e., before transitioning to the next state with the arrival of $i_{k+1}$). Starting from the current state $S_k = (t_k, I_k, \Psi_k, t^s_k, \Theta_k, t^d_k, t^r_k)$ and applying the decision $x_k = (x^{(1)}_k, x^{(2)}_k ) = (\Psi_k^x, t_k^{s,x}, \Theta_k^x, t^{d,x}_k)$, the resulting post-decision state is defined as $S_k^x = (t_k, I_k, \Psi_k^x, t^{s,x}_k, \Theta_k^x, t^{d,x}_k, t^r_k)$.

\paragraph{Stochastic Information}
The stochastic information $W_{k+1} \in \Omega$ consists of the new order $i_{k+1}$, which includes the order placement time $t^o_{i_{k+1}}$, the order location $l_{i_{k+1}}$, and possibly other relevant attributes specific to the problem. In a special case where no additional orders are received, represented by $W_{k+1} = \emptyset$, the decision process advances to the final state at \(T^c\).

\paragraph{Transition}
Given a state \( S_k = (S_k^{(1)}, S_k^{(2)}) \), a decision \( x_k = (x_k^{(1)}, x_k^{(2)}) \), and the stochastic information \( W_{k+1} \), the next state \( S_{k+1} = (S_{k+1}^{(1)} , S_{k+1}^{(2)}) \) is determined by the transition function \( S^M(S_k, x_k, W_{k+1}) \).
During the transition, the current time is updated to \( t_{k+1} = t^o_{i_{k+1}} \). The new set of open orders, \( I_{k+1} \), includes all orders from the updated trip plans \( \Theta^x_k \) with departure times later than \( t_{k+1} \), as well as the new order \( i_{k+1} \).
The resource schedules in \( S_{k+1}^{(1)} \) for each resource \( c \in C \) and the vehicle trip plans in \( S_{k+1}^{(2)} \) for each vehicle \( v \in V \) are then updated to reflect only the orders in \( I_{k+1} \).
Correspondingly, \( t^m_{k+1} \) in \( S_{k+1}^{(2)} \) is updated to include the end times of preparation for all open orders in \( I_{k+1} \) that began preparation prior to \( t_{k+1} \).
The return time \( t^{r}_{k+1,v} \) in \( S_{k+1}^{(2)} \) for each vehicle \( v \in V \) is updated to reflect the return time from its last trip departing before \( t_{k+1} \), or it is set to \( t_{k+1} \) if the vehicle is idle at that time.

\paragraph{Objective}
A policy, denoted by \( \pi \in \Pi \), is a mapping that assigns a feasible decision \( \pi(S_k) = x_k = (x_k^{(1)}, x_k^{(2)}) \in \mathcal{X}(S_k) \) to each state \( S_k \in \mathcal{S} \). Specifically, let \( \Pi^{(1)} \) and \( \Pi^{(2)} \) represent the sets of corresponding feasible policies for the first and second stages, respectively. A first-stage policy \( \pi^{(1)} \in \Pi^{(1)} \) is defined by the mapping \( \pi^{(1)}(S_k^{(1)}) = x_k^{(1)} \in \mathcal{X}^{(1)}(S_k^{(1)}) \), while a second-stage policy \( \pi^{(2)} \in \Pi^{(2)} \) is defined by \( \pi^{(2)}(S_k^{(2)}) = x_k^{(2)} \in \mathcal{X}^{(2)}(S_k^{(2)}) \). Together, a complete policy \( \pi = (\pi^{(1)}, \pi^{(2)}) \) is feasible and belongs to \( \Pi \) if and only if, for each state \( S_k = (S_k^{(1)},S_k^{(2)})\), \( \pi^{(1)}(S_k^{(1)}) \text{ and } \pi^{(2)}(S_k^{(2)})\) are compatible decisions. Then, we also refer to $\pi^{(1)} \text{ and } \pi^{(2)}$ as \emph{compatible policies}.

Starting from the initial state \( S_0 = (0, \emptyset, \emptyset, \emptyset, \emptyset, \emptyset, \emptyset, \emptyset) \), the goal is to find a policy \( \pi^* = (\pi^{(1)*}, \pi^{(2)*}) \) that minimizes the expected cumulative marginal costs across all decision points:
\begin{equation}\label{eq:optimal policy}
   \pi^* = (\pi^{(1)*}, \pi^{(2)*}) \in \argmin_{\pi = (\pi^{(1)}, \pi^{(2)}) \in \Pi} \mathbb{E}\left[\sum_{k=0}^K \mathbb{C}(S_k^{(2)}, \pi^{(2)}(S_k^{(2)})) \mid S_0\right].
\end{equation}
The optimal value of a state \( S_k \), denoted by \( V(S_k) \), represents the minimum expected future cost starting from \( S_k \) and can be expressed recursively as:
\begin{equation} \label{eq: V opt}
    V(S_k) = \min_{x_k=(x_k^{(1)}, x_k^{(2)}) \in \mathcal{X}(S_k)}\left\{ \mathbb{C}(S_k^{(2)}, x_k^{(2)}) + \mathbb{E}[V(S_{k+1}) \mid S_k, x_k]\right\},\quad \forall k \leq K,
\end{equation}
where $V(S_{K}) = \mathbb{C}(S_K^{(2)}, x_K^{(2)}), \text{ for all } S_K \text{ and } x_K=(x_K^{(1)}, x_K^{(2)}) \in \mathcal{X}(S_K)$.
The term \( \mathbb{E}[V(S_{k+1}) \mid S_k, x_k] \) represents the expected value of \( V(S_{k+1}) \) given state \( S_k \) and decision \( x_k \) (the cost-to-go) with the expectation taken over \( W_{k+1} \) (the scenario at \( t_{k+1} \)).
Equation (\ref{eq: V opt}) illustrates the dependency of each immediate cost on \( S_k^{(2)} \) and \( x_k^{(2)} \) only, while \( x_k^{(1)} \) impacts future costs through the transition function that determines $S_{k+1}$.

\section{Policy Decomposition}\label{sec: decomposition}
In this section we propose a novel decomposition
formulation to address the computational complexity of solving the DTS-OSP.
Unlike traditional decomposition methods designed for static or deterministic problems, this formulation accommodates the complexity of dynamic, stochastic systems where the solution is defined as a policy.

\paragraph{Preliminaries} Recall that the policies $\pi^{(1)} \in \Pi^{(1)}$ and $\pi^{(2)} \in \Pi^{(2)}$ govern the preparation and delivery stages, respectively, satisfying the stage-specific constraints $\mathcal{Q}_{\text{prep}}(S^{(1)}_k)$ and $\mathcal{Q}_{\text{deliv}}(S^{(2)}_k)$ for all states $(S^{(1)}_k, S^{(2)}_k) \in \mathcal{S}$. Notably, $\pi^{(1)}$, and $\pi^{(2)}$, may not have a compatible policy in the other stage.
Therefore, we define $\Tilde{\Pi}^{(2)} = \{\pi^{(2)} \in \Pi^{(2)} \mid \exists \pi^{(1)} \in \Pi^{(1)} \text{ such that } (\pi^{(1)}, \pi^{(2)}) \in \Pi \}$ as the set of second-stage policies that have a compatible first-stage policy. 
Given the definition of $\Tilde{\Pi}^{(2)}$, we now introduce an alternative problem representation to leverage the structure of the DOFP class.

\paragraph{Alternative Problem Representation}
The optimal delivery-stage policy can be defined by Eq.\eqref{eq: alternative problem formulation}.  
\begin{equation}\label{eq: alternative problem formulation}
   \pi^{(2)*} = \argmin_{\pi^{(2)} \in \Tilde{\Pi}^{(2)}} \mathbb{E} \left[\sum_{k=0}^K \mathbb{C}(S_k^{(2)}, \pi^{(2)}(S_k^{(2)})) \mid S_0 \right]
\end{equation}

This representation is motivated by the special structure of problems belonging to the DOFP class, in which the costs depend only on the second-stage policy. 
From \eqref{eq: alternative problem formulation}, we observe that the objective value is supposedly independent of $\pi^{(1)}$ and is only a function of $\pi^{(2)}$. 
However, $\pi^{(1)}$ has a role in determining the set of $\Tilde{\Pi}^{(2)} \subseteq {\Pi}^{(2)}$, which might be challenging to do in advance. 
Thus, we suggest a decomposition that leverages the two-stages structure of the problem, as detailed below.

\paragraph{Policy Decomposition}
We separate the optimization tasks into a Policy Master-Problem ($\mathcal{PMP}$), which focuses on delivery policies, and a Policy Sub-Problem ($\mathcal{PSP}$), which ensures compatibility with the preparation policies. Together, they iteratively refine the solution space until optimality is achieved. 

The $\mathcal{PMP}$ is defined as determining a delivery-stage policy $\pi^{(2)*}_{\text{{\tiny{MP}}}}$
as follows:
\begin{equation} \label{eq: MP policy}
 \mathcal{PMP}: \quad  \pi^{(2)*}_{\text{{\tiny{MP}}}} = \argmin_{\pi^{(2)} \in \Pi^{(2)}_{\text{{\tiny{MP}}}}} \mathbb{E} \left[\sum_{k=0}^K \mathbb{C}(S_k^{(2)}, \pi^{(2)}(S_k^{(2)})) \mid S_0 \right],
\end{equation}
where $\Pi^{(2)}_{\text{{\tiny{MP}}}}$ is a set of delivery-stage decisions in ${\Pi}^{(2)}$ that were not shown to be infeasible at a given iteration, hence, $\Tilde{\Pi}^{(2)} \subseteq \Pi^{(2)}_{\text{{\tiny{MP}}}} \subseteq {\Pi}^{(2)}$.
The initial set $\Pi^{(2)}_{\text{{\tiny{MP}}}}$ includes all delivery-stage policies $\Pi^{(2)}$ but can be later pruned according to the $\mathcal{PSP}$. 
The solution $\pi^{(2)*}_{\text{{\tiny{MP}}}}$ minimizes the objective value, but it does not necessarily ensure the existence of a compatible policy in $\Pi^{(1)}$, i.e., feasibility.

The $\mathcal{PSP}$, given $\pi^{(2)*}_{\text{{\tiny{MP}}}}$, searches for a compatible preparation-stage policy $\pi^{(1)}_{\text{{\tiny{SP}}}} \in \Pi^{(1)}_{\text{{\tiny{SP}}}}(\pi^{(2)*}_{\text{{\tiny{MP}}}}) = \{\pi^{(1)} \in \Pi^{(1)} \mid \pi^{(1)} \text{ and } \pi^{(2)*}_{\text{{\tiny{MP}}}} \text{ are compatible} \}$, where $\Pi^{(1)}_{\text{{\tiny{SP}}}}(\pi^{(2)*}_{\text{{\tiny{MP}}}})$ denotes the set of compatible preparation-stage policies to $\pi^{(2)*}_{\text{{\tiny{MP}}}}$. 
The problem is defined as:
\begin{equation} \label{eq: SP policy}
 \mathcal{PSP}: \quad     \text{Find } \pi^{(1)}_{\text{{\tiny{SP}}}} \in \Pi^{(1)}_{\text{{\tiny{SP}}}}(\pi^{(2)*}_{\text{{\tiny{MP}}}}).
\end{equation}
Note that $\mathcal{PSP}$ is a feasibility problem and it either obtains a feasible $\pi^{(1)}_{\text{{\tiny{SP}}}}$ (then $\pi^{(2)*}_{\text{{\tiny{MP}}}}$ is optimal) or concludes that there is no feasible solution, and hence $\pi^{(2)*}_{\text{{\tiny{MP}}}}$ (and possibly other second-stage policies) is excluded from the set $\Pi^{(2)}_{\text{{\tiny{MP}}}}$. Then, the $\mathcal{PMP}$ should be solved again with the updated $\Pi^{(2)}_{\text{{\tiny{MP}}}}$.

{Unlike classical decomposition methods, which generates optimality cuts to form lower bounds for a relaxed master problem, our policy-based decomposition framework does not rely on explicit dual information or linear relaxations. Instead, the refinement of the master policy space is driven by two components: (i) feasibility checks via the subproblem, which prune incompatible policies, and (ii) the cost-to-go which guides selection among candidate delivery-stage decisions. While these mechanisms do not yield formal bounds in the classical sense, they enable informed and structured exploration of the policy space under dynamic uncertainty.}

\begin{claim}[The iterative approach obtains an optimal policy]\label{th:The iterative approach converges to the optimal}
Provided that the sets $\Pi^{(1)}$ and $\Pi^{(2)}$ are finite, the iterative approach described above obtains an optimal policy $\pi^* = (\pi^{(1)*}, \pi^{(2)*})$ that satisfies Eq.\eqref{eq:optimal policy} (assuming that $\Pi \neq \emptyset$).
\end{claim}

\begin{proof}
Each refinement of $\Pi^{(2)}_{\text{{\tiny{MP}}}}$ ensures a strictly smaller feasible set, and because $\Pi^{(2)}$ is finite, our approach will terminate after a finite number of steps. Upon termination, the resulting pair $(\pi^{(1)}_{\text{{\tiny{SP}}}}, \pi^{(2)*}_{\text{{\tiny{MP}}}})$ is optimal because $\mathcal{PMP}$ minimizes the cost over the refined feasible set, and $\mathcal{PSP}$ ensures feasibility.
\end{proof}

Building on the theoretical basis above, we demonstrate how the $\mathcal{PMP}$ and $\mathcal{PSP}$ can be practically solved within a sequential decision-making framework, as common in dynamic optimization (e.g., via Bellman equations).
To this end, we define $\mathcal{X}_{\text{{\tiny{MP}}}}$ and $\mathcal{X}_{\text{{\tiny{SP}}}}$ below as the sets of delivery and preparation decisions that correspond to $\Pi^{(2)}_{\text{{\tiny{MP}}}}$ and $\Pi^{(1)}_{\text{{\tiny{SP}}}}$, respectively.
\begin{equation}
\mathcal{X}_{\text{{\tiny{MP}}}}(S_k^{(2)}) =  \{ \pi^{(2)}(S_k^{(2)}) \mid  \pi^{(2)} \in \Pi_{\text{{\tiny{MP}}}}^{(2)} \}, \quad \forall S_k^{(2)}  
\end{equation}
\begin{equation}\small
\mathcal{X}_{\text{{\tiny{SP}}}}(S_k,x_k^{(2)}) =  \{ \pi^{(1)}(S_k^{(1)}) \mid   \pi^{(1)} \in \Pi^{(1)}_{\text{{\tiny{SP}}}}(\pi^{(2)}_{\text{{\tiny{MP}}}}), \text{ } \pi^{(2)}_{\text{{\tiny{MP}}}} \in \Pi_{\text{{\tiny{MP}}}}^{(2)} \text{ such that } \pi^{(2)}_{\text{{\tiny{MP}}}}(S_k^{(2)}) = x_k^{(2)}\} , \quad \forall S_k, x_k^{(2)} 
\end{equation}
The presentation of the latter can also be simplified to $\mathcal{X}_{\text{{\tiny{SP}}}}(S_k,x_k^{(2)})= \{x_k^{(1)} \in \mathcal{X}^{(1)}(S_k^{(1)}) \mid (x_k^{(1)}, x_k^{(2)}) \text{ are compatible} \}$.

Then, solving $\mathcal{PMP}$ can be done by addressing an alternative representation of the $\mathcal{PMP}$, referred to as $\mathcal{MP}$, defined as follows.

\paragraph{Master-Problem ($\mathcal{MP})$:} 
Find $\pi^{(2)*}_{\text{{\tiny{MP}}}}$, the solution to $\mathcal{PMP}$, by solving:
\begin{equation}\label{eq: MP seq}
\mathcal{MP}: \quad \pi^{(2)*}_{\text{{\tiny{MP}}}}(S_k^{(2)}) \in \argmin_{x_k^{(2)} \in \mathcal{X}_{\text{{\tiny{MP}}}}(S_k^{(2)})} \{\mathbb{C}(S_k^{(2)}, x_k^{(2)}) + V_{\text{{\tiny{SP}}}}(S_k,x_k^{(2)})\}, \quad \forall S_k = (S_k^{(1)}, S_k^{(2)}),
\end{equation}

\noindent
where the second component of the objective, \(V_{\text{{\tiny{SP}}}}(S_k, x_k^{(2)})\), represents the expected future costs obtained by solving the Sub-Problem (\(\mathcal{SP}\)), as defined below.
%($V(S^x_k)$ for $x_k^(1)$ and $x_k^(2)$ obtained by \(\mathcal{MP}\) and \(\mathcal{SP}\)) 
Notably, while the minimization is performed over \(x_k^{(2)}\), the term \(V_{\text{{\tiny{SP}}}}\) introduces a dependency on \(x_k^{(1)}\), coupling the two decision stages.
The formulation in \eqref{eq: MP seq} solves the $\mathcal{PMP}$ by incorporating dynamic state transitions and future cost evaluations, aligning with the sequential decision-making framework. 

\paragraph{Sub-Problem ($\mathcal{SP}$):}
The Sub-Problem defined according to the sequential decision making framework, searches for a solution that is not only a compatible first-stage policy but also accounts for the expected future cost through its impact on the feasible delivery-stage policies. 
Thus, 
\begin{equation}\label{eq: SP seq}
\mathcal{SP}: \quad V_{\text{{\tiny{SP}}}}(S_k,x_k^{(2)}) = \min_{x_k^{(1)} \in \mathcal{X}_{\text{{\tiny{SP}}}}(S_k,x_k^{(2)})}
\mathbb{E}[V(S_{k+1}) \mid S_k, x_k=(x_k^{(1)}, x_k^{(2)})]
\end{equation}
If $\mathcal{X}_{\text{{\tiny{SP}}}}(S_k,x_k^{(2)}) = \emptyset$, the corresponding \(x_k^{(2)}\) is not feasible, and it is required to generate feasibility cuts to prune it (and possibly other decisions) from $\mathcal{X}_{\text{{\tiny{MP}}}}(S_k^{(2)})$.
This sequential decision making formulation provides an optimal solution, but it now requires evaluating the effect of $x_k^{(1)}$ on future costs rather than evaluating only policies $\pi_2$. 

To conclude, our decomposition method decouples the problem into a Policy Master-Problem and a Policy Sub-Problem. By that, it systematically reduces computational efforts, solving smaller, more focused problems iteratively rather than the entire problem at once. For convenience, \ref{app:notation} summarizes the notation presented thus far.

\section{The Decomposition-Driven Solution Framework}\label{sec: solution method description}

This section introduces the \textit{DDF-VFA}, developed to efficiently solve the DOFP by leveraging the decomposition structure outlined in Section \ref{sec: decomposition}. The framework is composed of two solution approaches: one targeting the \(\mathcal{MP}\) in \eqref{eq: MP seq}, and the other addressing the \(\mathcal{SP}\) in \eqref{eq: SP seq}, for each state \(S_k = (S_k^{(1)}, S_k^{(2)})\).

A main challenge in solving the \(\mathcal{MP}\) is the vast decision space, \(\mathcal{X}_{\text{{\tiny{MP}}}}(S_k^{(2)})\). Thus, an efficient solution mechanism is required to balance computational efforts and solution quality, as these decisions determine both immediate and future costs. Building on the concept of partial solutions introduced by \citet{NeriaGhost2024}, where partial solutions were explored across multiple stages without decomposition, our approach focuses exclusively on partial solutions for the second stage. Specifically, we propose a heuristic method based on an LNS (large neighborhood search) algorithm to explore partial delivery decisions (sequences of trips excluding departures scheduling and vehicle assignments) while completing them using a separate innovative algorithm, as detailed in \ref{sec: Heuristic to MP}.

Solving the $\mathcal{SP}$, given a decision $x^{(2)}_k$, presents significant challenges. To tackle this, we propose a fast heuristic approach, named the Synchronized Preparation Scheduling (SPS) algorithm, designed to efficiently generate a feasible decision,  $x_k^{(1)} \in \mathcal{X}_{\text{{\tiny{SP}}}}(S_k, x_k^{(2)})$. If no feasible $x_k^{(1)}$ is found, the algorithm generates feasibility cuts to refine the $\mathcal{MP}$ (Claim \ref{th:The iterative approach converges to the optimal}).
An additional challenge in \eqref{eq: SP seq} lies in the computational effort required to evaluate the cost $\mathbb{E}[V(S_{k+1}) \mid S_k, x_k=(x_k^{(1)}, x_k^{(2)})]$ for each considered compatible decision $x^{(1)}_k$. To address this, we incorporate a VFA (value function approximation) using neural networks (NNs) \citep{NeriaTzur2022}. 

Because the first-stage decision influences the objective value only indirectly through its impact on future second-stage feasible decisions, the SPS employs a fast algorithm that prioritizes identifying feasible first-stage decisions without exhaustively optimizing them. By that, it enables saving limited real-time computational resources to second-stage delivery decisions in the $\mathcal{MP}$. That is, the \textit{DDF-VFA} leverages the structure of the DOFP problem class, where costs are incurred in the second stage and envisions an optimization framework that focuses on these second-stage decisions.

\begin{figure}[t!]
    \centering
    \caption{The DDF-VFA overview. {Contrary to the benchmark that jointly searches, the combined $(x_k^{(1)}, x_k^{(2)})$ decision space, DDF-VFA delegates $x_k^{(1)}$ to a feasibility-only subproblem (SPS) and concentrates the search effort on $x_k^{(2)}$, where all variable cost is incurred.}}    \scalefigure{figures/two stage solution overview figure.pdf}
    \label{fig:framework_interaction}
\end{figure}
The \textit{DDF-VFA} facilitates a structured interaction between the \(\mathcal{MP}\) and the \(\mathcal{SP}\), iteratively refining the solution space to ensure compatibility between preparation and delivery stages while optimizing computational efficiency. Figure~\ref{fig:framework_interaction} illustrates this workflow. At each iteration, the framework generates a delivery-stage decision \(x_k^{(2)}\) and searches for a compatible preparation-stage decision \(x_k^{(1)}\). If a feasible \(x_k^{(1)}\) is identified, the combined decision \( (x_k^{(1)}, x_k^{(2)}) \) is evaluated based on its immediate cost, \(\mathbb{C}(S_k^{(2)}, x_k^{(2)})\), and its estimated future cost, denoted \(\hat{V}(S_k^x)\), given by the VFA. This iterative process continues until a stopping criterion is met, yielding the best decision found \(x_k^*\) for state \(S_k\).
In the following subsections, we provide more details on the solution approaches for the \(\mathcal{MP}\) (\ref{sec: Heuristic to MP}) and the \(\mathcal{SP}\) (\ref{sec: Heuristic to SP}). A psuedocode is provided in \ref{app: Pseudocode for DDF-VFA}.

\subsection{A Solution Approach to the MP}\label{sec: Heuristic to MP}
\noindent To address the $\mathcal{MP}$ we develop an LNS algorithm which creates \emph{partial delivery decisions}, denoted as \(\boldsymbol{\theta} = (\theta_{(1)}, \theta_{(2)}, \dots)\), representing sequences of trips across all vehicles. These partial delivery decisions do not include the specific vehicle assignments and departure times. For the example in Figure~\ref{fig:example}, such partial delivery decisions may be $((1,4,2),(3,6,5))$ or $((1,4),(2,5),(3,6))$, etc. The latter means that first a vehicle starts a trip $\theta_{(1)}$ to deliver orders 1 and 4, then a vehicle starts a trip $\theta_{(2)}$ to deliver orders 2 and 5, and finally a vehicle starts a trip $\theta_{(3)}$ to deliver orders 3 and 6. The exact vehicle assignments and departure times remain unspecified, and then each partial delivery decision is passed to a dedicated algorithm, the Trip Assignment and Scheduling (TAS), which complements those partial delivery decisions to $x_k^{(2)} \in \mathcal{X}_{\text{{\tiny{MP}}}}(S_k^{(2)})$. 
 
\subsubsection{The Large Neighborhood Search.}\label{sec: LNS}

\noindent The LNS framework starts with an initial decision, based on appending the new order to the plan in $S_k$, using a first-in-first-out heuristic, referred to as the \textit{FIFO} benchmark, which we present later in Section \ref{sec: benchmarks}. The LNS then employs problem-specific operators in each iteration to explore the space of partial delivery decisions {(descriptions of the operators can be found in ~\ref{app: LNS Adjustments})}. These operators modify the grouping of orders into trips, the sequence of order visits within each trip, and the sequence of trip departures. During the generation process, the LNS ensures that each trip is at least \emph{feasible in isolation}, meaning it adheres to basic constraints such as capacity limits.

\subsubsection{The Trip Assignment and Scheduling Algorithm.} \label{sec: the TAS}

The TAS transfers the partial delivery decision $\boldsymbol{\theta}$ into a (complete) second-stage decision $x^{(2)}_k$. First, it checks whether $\boldsymbol{\theta}$ can potentially become a decision $x^{(2)}_k \in \mathcal{X}_{\text{{\tiny{MP}}}}(S_k^{(2)})$. If it can, then it sequentially generates departure times and vehicle assignments for each trip in $\boldsymbol{\theta}$, where after each such assignment, it verifies feasibility, given the current state, for the already scheduled trips with a potential compatible decision $x_k^{(1)}$, see below the detailed steps. 

\paragraph{Initial Check}
The first step is to verify whether the partial delivery decision \(\boldsymbol{\theta}\) satisfies the delivery-stage constraints \(\mathcal{Q}_{\text{deliv}}(S_k^{(2)})\). This check identifies any obvious conflicts, such as orders necessitating departure times that violate vehicle availability. 
If \(\boldsymbol{\theta}\) fails this check, it is deemed infeasible, and the LNS generates a new partial delivery decision. If the check passes, the algorithm proceeds to the next phase.

\paragraph{Assigning Departure Times}  
Departure times are assigned sequentially for each trip in \(\boldsymbol{\theta}\). Let \(l = 1, \dots, |\boldsymbol{\theta}|\) represent the index of trips in \(\boldsymbol{\theta}\). The algorithm follows the following steps:
\begin{enumerate}
    \item \textit{Current Trip Assignment:} The departure time of the current trip \(\theta_{(l)}\) is scheduled as early as possible, considering the availability of vehicles and the restrictions of orders in the trip, e.g., the lower bounds on orders' ready times (after preparation).
    \item \textit{First-Stage and Synchronization Feasibility:} Once a departure time for the current trip is assigned, the potential of reaching a decision $x_k^{(2)}$ that has a compatible decision $x_k^{(1)}$ is evaluated by considering all orders in trips \(\theta_{(1)}, \dots, \theta_{(l)}\). That is, the feasibility of synchronizing these trips with the preparation-stage is verified for their departure times that were set in Step 1. This check is performed using the SPS algorithm (see Section~\ref{sec: the SPS}). If feasibility is confirmed, the TAS proceeds to the next trip.
    \item \textit{Postponing Departure Times:} If feasibility is not achieved, the departure time for the current trip is incrementally postponed (e.g., by one time-unit), and feasibility is re-evaluated in Step 2. 
\end{enumerate}
The postponement in Step 3 may cascade to subsequent trips, as \(\boldsymbol{\theta}\) specifies a sequence of trips. If \(\theta_{(l)}\) requires a later departure time, all trips \(\theta_{(l+1)}, \dots, \theta_{(|\boldsymbol{\theta}|)}\) must also depart no earlier than the adjusted time of \(\theta_{(l)}\). This cascading adjustment eliminates from consideration a larger set of infeasible policies, streamlining the search process.
If this postponement renders \(\boldsymbol{\theta}\) infeasible, the TAS terminates, and the LNS generates a new partial decision. Otherwise, the algorithm produces a feasible complete decision \(x_k^{(2)}\) for evaluation.

{The TAS can straightforwardly accommodate several optional features as small additions to its initial check, departure-assignment, or vehicle-assignment steps; the only case when a substantive extension to TAS is required is when the objective function minimizes soft-delivery-window violations, which leads to a DP-based refinement. \ref{app: TAS feature extensions} details examples of several per-feature adjustments (including the DP-refinement).}

\subsection{A Solution Approach to the SP}\label{sec: Heuristic to SP}

The \(\mathcal{SP}\) obtains the delivery-stage decision \(x_k^{(2)}\) found by \(\mathcal{MP}\) in \ref{sec: Heuristic to MP}. Towards that, we describe the SPS algorithm that aims to generate feasible preparation-stage decisions efficiently and a VFA framework to evaluate decisions. These components are described in the following subsections.

\subsubsection{Synchronized Preparation Scheduling Algorithm.} \label{sec: the SPS}
The SPS algorithm determines a compatible preparation-stage decision \(x^{(1)}_k \in \mathcal{X}_{\text{{\tiny{SP}}}}(S_k, x_k^{(2)})\), or asserts that no compatible decision exists. To do so, it first computes feasible starting preparation time windows \([a_i, b_i]\) for each order \(i \in I_k^{(1)}\). These windows are derived solely from the synchronization constraints \(\mathcal{Q}_{\text{sync}}(S_k)\). For example, we ensure that preparation is completed before the associated delivery trip departs in \(x_k^{(2)}\) by setting \(b_i = t^d_{\theta(i)} - t_i^p\), where \(t^d_{\theta(i)}\) is the departure time of the trip containing order \(i\). Within these time windows, the solution $x^{(1)}_k$ needs to satisfy only preparation-stage constraints, \(\mathcal{Q}_{\text{prep}}(S_k^{(1)})\).
We obtain a MILP formulation, which can be solved optimally or heuristically by traditional operations research tools. The following steps summarize our algorithm.

First, select order $i$ from the set of all unassigned orders according to earliest due date (EDD), break a tie by SPT. Then, attempt to assign order $i$ to the resource that minimizes the setup time from its last assigned order. For this assignment to be considered, the algorithm checks whether it is possible to feasibly assign the next two orders to any available resources after assigning order $i$. This lookahead step ensures that the selected machine and its updated availability after order $i$ can still accommodate at least two future jobs without violating their due dates.
If no resource can feasibly assign order $i$ while maintaining the feasibility of the next two orders, return "infeasible." If the assignment is successful, update the availability of the selected resource, and repeat the process for the next order, using the same procedure for assigning jobs and checking future feasibility.

{Note that when no sequence-dependent setups exist between orders the procedure reduces to a standard EDD-based assignment heuristic (without lookahead). Concretely, it is equivalent to the H1 method of \citet{ho1995minimizing}, originally proposed for parallel-machine scheduling: orders are sorted by EDD (ties broken by shortest processing time, SPT) and each order \(i\) is assigned to the busiest available resource, i.e., the one with the latest availability that can complete it without tardiness; if no such resource exists, the instance is declared infeasible. Otherwise the resource's availability is updated and the next order is processed.}

{Beyond the implementation described above, the SPS accommodates several other optional features either as a parameter-style update to its scheduling rule (heterogeneous resources) or as a full algorithmic substitution of the SPS itself (e.g., batched preparation or multi-stage preparation). See \ref{app: SPS feature extensions} for more details.}

The SPS results in a preparation decision $x_k^{(1)}$, which together with the $\mathcal{MP}$ outcome, $x_k^{(2)}$, form the decision $x_k$. This decision is used to approximate the future cost as described next.

\subsubsection{Future Cost Approximation.}\label{sec: VFA} 
In Eq.\eqref{eq: SP seq}, the value of each considered $x_k^{(1)}$, $\mathbb{E}[V(S_{k+1}) \mid S_k, x_k]$, denoted as \( V(S_k^x)\), represents the cost-to-go of the post-decision state $S_k^x$, obtained from \( S_k \) and \( x_k = (x_k^{(1)}, x_k^{(2)})\).
% The value of a post-decision state, denoted as \( V(S_k^x) = \mathbb{E}[V(S_{k+1}) \mid S_k, x_k] \), represents the cost-to-go given \( S_k \) and \( x_k \). 
In the \textit{DDF-VFA}, the exact cost-to-go is computationally prohibitive to calculate within a reasonable time frame. Hence, we employ NN-based VFA for fast estimation.  NNs are able to approximate complex, nonlinear relationships between post-decision-states and their future costs, making them well-suited for the DOFP. The VFA predicts expected costs given aggregated post-decision states, which summarize key details such as resources and vehicle availability, along with the current time \( t_k \), thus enabling scalability and transfer learning. The VFA is refined offline through training, iteratively minimizing prediction errors between estimated and observed costs in a simulated environment using RL. Once trained, the model is applied online to evaluate decisions, enabling efficient decision-making and dynamic schedule updates.
The VFA output, denoted as \(\hat{V}(S_k^x)\), integrates seamlessly into the \textit{DDF-VFA}, enhancing its ability to balance immediate and future costs effectively. Implementation details, including feature selection and architecture, are consistent with \citet{NeriaGhost2024}.

\section{Dynamic Fulfillment Problem Variants}\label{sec:model_problems}
We present two problem variants as case studies within the DOFP class. These variants were chosen to demonstrate the versatility of the DOFP class across diverse domains by addressing real-world challenges and highlighting structural differences.
Each variant emphasizes a distinct operational focus: the Picking Variant minimizes travel time and penalized delay; and the Production Variant minimizes soft time window violations while incorporating sequence-dependent setup constraints. Together, these variants showcase the framework’s adaptability to diverse objectives, constraints, and decision spaces.

To formally ground these distinct operational characteristics within the DOFP framework, we describe a MILP formulation to rigorously define and bound the feasible decision space $\mathcal{X}(S_k)$ given an individual system state $S_k$ at decision epoch $k$. In particular, the below MILP defines the feasible preparation decisions \(X^{(1)}(S_k^{(1)})\), delivery decisions \(X^{(2)}(S_k^{(2)})\), synchronization constraints \(\mathcal{Q}_{\text{sync}}(S_k)\), and the objective function \(J(\Theta, t^d)\) introduced abstractly in \ref{sec: MDP}. While the theoretical breadth of the broader DOFP class makes a single, exhaustive MILP formulation that captures every possible permutation mathematically impractical, the unified MILP below is designed specifically to encompass the structural components of our two example variants.

\subsection{Formulation of the Decision Space at a Given State in the Example Variants}\label{sec: MILP formulation}
The constraints below are written in a feature-parameterized form allowing each example variant to be recovered by fixing the values of the parameters \(s_{ij}\), \(d_i\), and \((a_i, b_i, \rho_e, \rho_l)\) introduced below.  

\paragraph{Preparation Stage}
At each decision point \( k \), the preparation state \( S_k^{(1)} \) consists of all orders \( i \in I^{(1)}_k \), each characterized by a processing time \( t^p_i \). {The preparation resources form a set \(C\) of homogeneous/hetrogeneous units (the concrete role - e.g., pickers or production machines - is variant-specific).} The preparation assignment plan \( \Psi_k \) defines the sequence of orders allocated to each resource, with \( \Psi_{kc} \) denoting the assignments specific to resource \( c \in C \). The set \( t^s_k \) represents the start times for processing each order.

The feasible decision set \( X^{(1)}(S^{(1)}_k) \) is formulated as a MILP. We define a binary variable \( z_{ict} \in \{0,1\} \), where \( z_{ict} = 1 \) indicates that resource \( c \) starts processing order \( i \) at time \( t \). Note that the preparation decision \( x^{(1)}_k = (\Psi_k^x, t_k^{s,x}) \) can then be derived from the values of \( z_{ict} \) for all \( i \), \( c \), and \( t \). Constraints \eqref{con:2}–\eqref{con:6} define the set of first-stage constraints \( \mathcal{Q}_{\text{prep}}(S_k^{(1)}) \):
\begin{align}
    \sum_{c\in C} \sum_{t=t^o_i}^T z_{ict} &= 1, \quad \forall i \in I_k^{(1)} \quad \text{(order assigned once)} \label{con:2}\\
    \sum_{i\in I_{k}^{(1)}} z_{ict} &\leq 1, \quad \forall c \in C, t_k \leq t \leq T \quad \text{(resource capacity)} \label{con:3}\\
    \sum_{j\in I^{(1)}_{k}} \left({t^\prime} + t^p_j + s_{ji}\right) z_{jc{t^\prime}} &\leq T + z_{ict}(t -T), \quad \forall i, c, t_k \leq t^\prime < t \leq T \quad \text{(resource availability)} \label{con:4}\\
    z_{ict} &= 1, \quad \forall c \in C, i \in \Psi_{kc}, t \in \{t^s_{ik}\leq t_k \mid i\in I_k^{(1)}\} \label{con:5}\\
    z_{ict} &\in \{0,1\}, \quad \forall i \in I_k^{(1)}, c \in C, t_k \leq t \leq T \quad \text{(binary definition)} \label{con:6}
\end{align}
These constraints ensure that each order in \( I_k^{(1)} \) is assigned exactly once, with no overlap in assignments for any resource at any time, and that tasks already started remain unchanged. {In Constraint~\eqref{con:4} we introduce a setup time \(s_{ji}\).  Setting \(s_{ji}=0\) for all \(i,j\) recovers a no-setup form.} 

\paragraph{Delivery Stage}
At each decision point \( k \), the delivery state \( S_k^{(2)} \) characterizes each order \( i \in I^{(2)}_k \) by its location \( l_i \) and placement time \( t^o_i \). {Orders are transported by vehicles \(v \in V\) with a capacity of \( \kappa \) units, and each order \(i\) consumes \(d_i\) units of capacity.}

The feasible decision set \( X^{(2)}(S^{(2)}_k) \) is also formulated as a MILP. The binary variable \( y_{ivt} \in \{0,1\} \) indicates whether vehicle \( v \) departs with order \( i \) at time \( t \). Another binary variable, \( q_{ijt} \in \{0,1\} \), specifies whether orders \( i \) and \( j \) are delivered sequentially on a trip departing at time \( t \). The discrete variable \( t^a_i \in [t_k, T] \) represents the arrival time of order \( i \) at the customer's location, while \( t^r_{vt} \in [0, T] \) denotes the return time of vehicle \( v \) if dispatched at time \( t \).
Note that the delivery decision \( x^{(2)}_k = (\Theta_k^x, t_k^{d,x}) \) can then be derived from the values of these decision variables.
The constraints defining \( \mathcal{Q}_{\text{deliv}}(S_k^{(2)}) \) are:
\begin{alignat}{3}
    &\sum_{v\in V} \sum_{t=t_k}^T y_{ivt} = 1, & \quad \forall i \in I^{(2)}_k & \quad \text{(assign each order to one trip)} \label{con:16}\\
    &(y_{ivt} - 1) \cdot T + t^a_i + t^t_{i0} \leq t^r_{vt}, & \quad \forall i, v, t_k \leq t \leq T & \quad \text{(vehicle availability)} \label{con:7}\\
    &t^r_{vt'} \leq y_{ivt} \cdot t + T(1 - y_{ivt}), & \quad \forall i, v, t_k \leq t' < t \leq T & \quad \text{(time sequencing)} \label{con:8}\\
    &\sum_{i\in I^{(2)}_k} d_i \cdot y_{ivt} \leq \kappa, & \quad \forall v, t_k \leq t \leq T & \quad \text{(capacity constraints)} \label{con:9}\\
    &\sum_{i\in I^{(2)}_k} y_{ivt} - 1 \leq \sum_{j \in I^{(2)}_k, j \neq i} q_{jit}, & \quad \forall v, t_k \leq t \leq T & \quad \text{(trip sequence)} \label{con:10}\\
    &2 \cdot q_{jit} \leq y_{ivt} + y_{jvt}, & \quad \forall i, j, v, t_k \leq t \leq T & \quad \text{(trip inclusion)} \label{con:11}\\
    &(q_{ijt} - 1) \cdot T + t^t_{ij} \leq t^a_j - t^a_i, & \quad \forall i \neq j, t_k \leq t \leq T & \quad \text{(no sub-tours)} \label{con:12}\\
    &t \cdot y_{ivt} + t^t_{0i} \leq t^a_i, & \quad \forall i, v, t_k \leq t \leq T & \quad \text{(time feasibility)} \label{con:13}\\
    &y_{ivt}, q_{ijt} \in \{0,1\}, & \quad \forall i, j \in I^{(2)}_k, v \in V, t_k \leq t \leq T & \quad \text{(binary definitions)} \label{con:14}
\end{alignat}
Constraint \eqref{con:16} ensures each order in \( I_k^{(2)} \) is assigned to exactly one trip, while constraints \eqref{con:7}–\eqref{con:14} enforce vehicle availability, capacity, and sequence feasibility, eliminating sub-tours. {The demand size \(d_i\) in Constraint~\eqref{con:9} represents the general case; setting \(d_i = 1\) for all \(i\) recovers the unit-size form.}

\paragraph{Combined Decision Set}
The feasible decision set \( \mathcal{X}(S_k) \) has to satisfy the preparation constraints \( \mathcal{Q}_{\text{prep}}(S^{(1)}_k) \), the delivery constraints \( \mathcal{Q}_{\text{deliv}}(S^{(2)}_k) \), and the synchronization constraints \( \mathcal{Q}_{\text{sync}}(S_k) \), with the latter specified by Constraint \eqref{con:17}. This synchronization constraint ensures that orders are only delivered once the preparation is completed.
\begin{equation}\label{con:17}
\sum_{c\in C} \sum_{t=1}^T (t + t^p_i) \cdot z_{ict} \leq \sum_{v\in V} \sum_{t=1}^T t \cdot y_{ivt}, \quad \forall i \in I_k.
\end{equation}

\paragraph{The Objective}
{The objective we consider in our example variants is minimizing a combination of penalized total travel time and deviation of each delivery from a target service window. We write the cost in a unified, feature-controlled form, allowing both variants of the DOFP class to be recovered by setting specific parameters. To that end, define the travel time penanlty \(\rho_t\) and consider a target service window \([a_i, b_i]\) for each order \(i\), with per-time-unit earliness and lateness penalties \((\rho_e, \rho_l)\). The total cost of a given delivery plan \((\Psi, t^s, \Theta, t^d)\) is}
\begin{equation}\label{eq: unified objective}
{{J}(\Theta, t^d) = \rho_t\sum_{\theta \in \bigcup \Theta} \sum_{l=1}^{|\theta|-1} t^t_{i^\theta_l, i^\theta_{l+1}} \; + \sum_{\theta \in \bigcup \Theta} \sum_{j=1}^{|\theta|} \left[ \rho_e \max\!\big(0,\, a_{i^\theta_j} - \hat{t}^a_{\theta,j}\big) + \rho_l \max\!\big(0,\, \hat{t}^a_{\theta,j} - b_{i^\theta_j}\big) \right],}
\end{equation}
{where \(\hat{t}^a_{\theta,j} = t^d_\theta + \sum_{l=1}^{j-1} t^t_{i^\theta_l, i^\theta_{l+1}}\) denotes the arrival time of the \(j\)-th order on trip \(\theta\). The first term accounts for total travel time; the bracketed term penalizes earliness and lateness relative to each order's window.}
Then, the immediate cost reflects the marginal cost resulting from updating the plan in $S^{(2)}_k$, $(\Theta_k, t^d_k)$, to the new plan according to $x^{(2)}_k$, $(\Theta_k^x, t^{d,x}_k)$, and can be obtained by:
\begin{equation}\label{eq: immediate cost}
\mathbb{C}(S^{(2)}_k,x^{(2)}_k) = {J}(\Theta_k^x, t^{d,x}_k)  - {J}(\Theta_k, t^d_k)
\end{equation}
The sum of immediate costs across all observed states equals the total realized costs over the entire period.

We emphasize again that the above framework supports other constraints or objective functions as long as the seven conditions in \ref{sec:Class Description} are met. Additional supported {example} features (e.g., batching, heterogeneous resources, perishability constraints) are presented in~\ref{app: MILP extensions} as extensions to the MILP formulation.

\subsection{The Picking Variant}\label{sec: problem statement Picking Variant}
This variant follows \citet{d2024integrating} and models an order picking environment within a warehouse tightly coupled with outbound fleet delivery logistics, where individual vehicle trips serve multiple drop-off destinations sequentially. The primary objective is to minimize total vehicle travel times alongside penalized delivery tardiness. 

As detailed in the comprehensive literature review of Section~2.1, real-time operating costs in this specific layout are predominantly concentrated in the delivery routing stage, while the picking capacity remains constrained on a rolling workforce schedule (e.g., \citealp{neves2019solving, rijal2023dynamics}). Grounded in these industry structures, and motivated by the myopic, non-anticipatory approaches of \citet{rijal2023dynamics} and \citet{d2024integrating}, we compare our framework directly to an \textit{Integrated} benchmark policy that optimizes both fulfillment stages concurrently but lacks long-term stochastic foresight regarding future order arrivals.

{The Picking Variant is formally obtained from the unified mathematical formulation presented in Section~\ref{sec: MILP formulation} by enforcing the following parameter state: sequence-dependent setups are omitted (\(s_{ji}=0\) for all \(i,j\)), orders possess unit capacity requirements (\(d_i = 1\) for all \(i\)), and the preparation resources \(c \in C\) act as a homogeneous picker workforce. The service target simplifies to a one-sided commitment derived from a click-to-door promise interval \(\tau\), specializing the objective parameters to \(a_i = 0\), \(b_i = t^o_i + \tau\), \(\rho_t = 1\), \(\rho_e = 0\), and \(\rho_l := \rho>0\).}

\subsection{The Production Variant}\label{sec: problem statement Production Variant}
The second variant models a manufacturing-fulfillment pipeline inspired by the scheduling complexities studied in \citet{wu2022variable}. Here, heterogeneous items are processed on a set of homogeneous machines and immediately routed to customers under tight, soft time-window bounds. In contrast to the Picking Variant, this environment features highly constrained sequence-dependent setup switchovers and heterogeneous order volumes that fractionally exhaust vehicle capacities. The primary objective is to minimize total soft time-window violations.

Following the integrated production-routing paradigms reviewed in Section~2.1 (e.g., \citealp{berghman2023review, KERMANI2024}), this problem layout represents an operational setting where dynamic delivery penalties dominate final system performance, yet upstream combinatorial constraint matrices heavily dictate near-term execution feasibility. Evaluating our framework on this variant tests the scalability of the policy-level decomposition under continuous real-time order arrivals paired with severe changeover delays—a setup reflecting rapid-turnaround e-commerce fulfillment centers, modern restaurant kitchens, and custom manufacturing cells.

{The Production Variant is recovered directly from the unified formulation of Section~\ref{sec: MILP formulation} without simplifying structural assumptions. It explicitly retains the full sequence-dependent setup matrix \(s_{ij}\) and heterogeneous order demands \(d_i\). Each order \(i\) is bound symmetrically by an earliest and latest delivery target window \([t^{sw}_i, t^{ew}_i]\), which specializes the objective function parameters to \(a_i = t^{sw}_i\), \(b_i = t^{ew}_i\), \(\rho_t = 0\), and \(\rho_e = \rho_l = 1\), yielding total combined earliness and lateness violations as the underlying cost metric.}
\section{Computational Study}\label{sec: results}
In this section, we present our computational study to test our proposed policy and derive managerial insights. We first give an overview of the instances and benchmark policies in \ref{sec: instances} and \ref{sec: benchmarks}, respectively. We then analyze the objective values of all policies over the two problem variants in \ref{sec: results objective values all}. Further, we analyze problem-oriented aspects and investigate our design-choices in LNS, TAS, and SPS in \ref{sec:Method-Oriented Analysis}.
Additional results, including different instances, are described in \ref{app:additional_results}. %To further demonstrate the generalizability of our framework, we also present an extended case study on a third, priority-weighted pharmacy prescription delivery DOFP variant in~\ref{app: pharmacy variant}.

\subsection{Instances}\label{sec: instances}

We test each problem-variant on a different real-world data set to ensure the robustness of our approach across multiple problems and across multiple data sets, see ~\ref{app: instances}.

\noindent\textbf{Picking.} Based on the \citet{Meituan24} dataset, each instance includes five resources, ten vehicles, and an average of 170.5 daily orders. Orders arrive via a non-stationary Poisson process (with arrival rates changing every two hours) from 8:00 to 24:00, with processing times normally distributed (with a mean of 7 minutes and std of 1 minute). A 60-minute click-to-door delivery promise is applied, with a delay penalty factor of 10 per minute (equating the costs of one minute of delay to ten minutes of travel time).

\noindent\textbf{Production.} Adapted from \citet{ortec2022}, these instances feature on average 200–300 orders per day, with vehicles that have a capacity of 145 units. Orders, locations, and time windows are sampled from the dataset (of time horizon of about 700 minutes). Each order belongs to one of four classes, with setup times of 5 or 10 minutes between orders from different classes and zero between orders from the same class.

\subsection{Benchmark Policies}\label{sec: benchmarks}
We compared \textit{DDF-VFA} against the following four benchmark methods:

\noindent\textbf{(1) FIFO} serves as a baseline method, processing orders in the sequence they arrive without reordering or optimization, used in \citet{liu2022approximate} and as a benchmark in \citet{rijal2023dynamics}. This approach reflects what is often done in practice and helps establish a simple reference for evaluating more advanced methods. 

\noindent\textbf{(2) The Integrated approach} optimizes current operations without accounting for future uncertainties, as implemented in \citet{rijal2023dynamics, d2024integrating}. We implement the benchmark approach developed in \citet{NeriaGhost2024}, which uses an LNS that searches the entire space of partial decisions without assignment and timing via LNS. Assignment and scheduling to specific resources and vehicles are handled by a separate polynomial algorithm. This method is called \enquote{Integrated} because it considers both stages simultaneously.

\noindent\textbf{(3) The AI method} (introduced in \citet{NeriaGhost2024}) builds on the \textit{Integrated} method, and incorporates RL techniques through VFA to evaluate the cost-to-go of decisions, providing a more adaptive and forward-looking approach. The main difference between this method and the \textit{DDF-VFA} is how the decision space is searched. That is, while the \textit{DDF-VFA} employs a decomposition approach and focuses its optimization efforts on the second stage, the \textit{AI} optimizes both stages directly and simultaneously.

\noindent\textbf{(4) DDF-G} (with \enquote{G} denoting \enquote{greedy}) mirrors the structure of our \textit{DDF-VFA} but excludes evaluation via VFA. Decisions are evaluated solely based on immediate costs, without considering future implications. This method allows us to assess the impact of VFA on solution quality.

These benchmarks collectively highlight the trade-offs between simple heuristics, integrated optimization, 
decomposition, and advanced ML based techniques. 

\subsection{Performance of the DDF-VFA}\label{sec: results objective values all}
To evaluate the performance of the \textit{DDF-VFA}, we compare its objective value to the values obtained by the benchmark policies for the two problem variants. To this end, we run 100 realizations for each problem variant and calculate the respective average objective values. We then calculate the relative percentual cost reduction of the \textit{DDF-VFA} as follows:
{\small
\[
\text{\% Cost Reduction } = \left( \frac{\text{Value of Benchmark Policy} - \text{Value of DDF-VFA}}{\text{Value of Benchmark Policy}} \right) \times 100.
\]
}

\begin{figure}[t!] \centering \caption{Average relative cost reduction of \textit{DDF-VFA} across all benchmarks and example problem variants.} 
\includegraphics[width=0.6\textwidth]{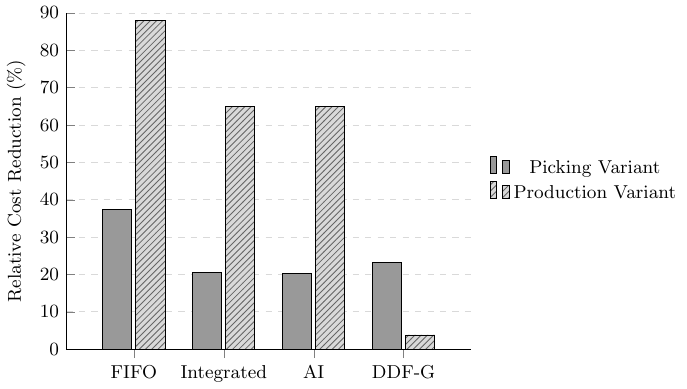}
%\scalefigure{figures/relative_improvement_variants.pdf} 
\label{fig: summary improvement} \end{figure}
\paragraph{Average improvements} The results are shown in Figure~\ref{fig: summary improvement}. We first observe that  \textit{DDF-VFA} improves upon all benchmarks for both problem variants, and the improvement is in the range of 3\% to 87\%. The largest average cost reductions are obtained relative to FIFO, followed by \textit{Integrated}, \textit{AI}, and finally \textit{DDF-G}. The magnitude of the improvement is not only due to the improved solution method, but it is also due to the problems' structure. For example, in warehouse picking operations, travel times are a significant and unavoidable cost factor. Thus, the objective values for all policies are large, and the relative improvement becomes smaller. Another notable observation is that \textit{DDF-G} outperforms \textit{AI} in the production variant. In this problem, the time windows complicate the search for effective delivery schedules. Thus, the focus of this search on the delivery stage (as done in \textit{DDF-VFA}) instead of on the entire decision space (as done in \textit{AI}) is particularly beneficial. The time windows may also be the reason for the relatively small improvement of \textit{DDF-VFA} over \textit{DDF-G} as they restrict the decision space, reduce flexibility, and, therefore, the need for explicit anticipation. The diminishing value of anticipation is also supported by the similar performance of \textit{Integrated} and \textit{AI}.

When comparing the two methods that employ anticipation, \textit{AI} and \textit{DDF-VFA}, under similar runtime conditions and the same number of LNS search iterations, \textit{DDF-VFA} consistently yielded higher-quality solutions.
The \textit{AI} approach struggled due to the simultaneous optimization of both stages. These findings highlight the advantage of the decomposition, particularly in real time environments where computational resources are constrained.

To evaluate the impact of the VFA within \textit{DDF-VFA}, we compare its performance against \textit{DDF-G}. The results reveal that incorporating VFA enhances the solution quality especially in the Picking variant where anticipating future decisions is critical. {Quantitatively, in the Picking variant at the highest tested delay penalty ($\rho=50$, Table~\ref{tab: methods performance results Picking Variant}), \textit{DDF-VFA} reduces the total objective from 45{,}392 (under \textit{DDF-G}) to 36{,}024 minutes, a 20.6\% improvement, by trading a small increase in travel time for a substantial reduction in delay. The gain is more modest in the Production variant: soft time windows compress the effective decision space and restrict the recourse available to delay-sensitive decisions, leaving less room for anticipation to help. This is consistent with the diminishing gap between \textit{Integrated} and \textit{AI} in the same variant, which we already noted above.}

\subsection{Analysis of the \textit{DDF-VFA} Components and its Robustness}\label{sec:Method-Oriented Analysis}
We have presented a novel solution framework, \textit{DDF-VFA}, comprising several individual components. We have implemented this framework for two problem variants (Picking and Production), which differ in their objectives, constraints, and data characteristics.
In this section, we aim to provide a deeper analysis of the framework’s components and its adaptability across these diverse problem settings. Specifically, we analyze the contribution of problem-specific LNS operators in solving the $\mathcal{MP}$, and effective SPS algorithms in solving the $\mathcal{SP}$, to the overall performance. Our analysis highlights the importance of ensuring that all of the decomposition components work well together. Additionally, we investigate the framework's robustness by evaluating its response to changes in important, problem-specific characteristics in objective function parameters, problem constraints, and input data.

For this analysis, we select representative experiments that highlight the role of these components. This approach clarifies the interplay between methodological design and problem-specific challenges, offering insights into the scalability and applicability of \textit{DDF-VFA} to real-world logistics. The \textit{DDF-VFA} is retrained for each configuration tested in this analysis to ensure that the model fully adapts to the specific characteristics of the parameter set. We show that \textit{DDF-VFA} consistently outperforms all benchmark policies, highlighting its robustness and providing valuable insights. 

\subsubsection{{The Role of Neighborhood Operator Design in Master Policy Optimization.}}\label{sec: results for Picking Variant}
The combinatorial complexity of the master problem ($\mathcal{MP}$) policy search space necessitates neighborhood operators that directly align with the core dimensions of the underlying objective function. When a problem variant exhibits a heavy spatial-routing cost component, relying solely on temporal or deadline-oriented operators can severely restrict the local neighborhood search from discovering localized travel efficiencies. 
To isolate and evaluate this dependency within the \textit{DDF-VFA} framework, we conduct a structural operator-ablation study. We choose an evaluation environment characterized by a strong travel-time emphasis (instantiated here via the Picking Variant framework) and disable our custom travel-minimizing operators, leaving only standard deadline-oriented operators active during the LNS iterations. Our experiments demonstrate that restricting the master policy's search space topology in this manner results in an immediate 8.1\% increase in vehicle travel times. This significant degradation confirms that framework performance is highly sensitive to the presence of objective-targeted operators, which are required to successfully navigate local optima during master policy updates.

\subsubsection{{Sensitivity to Objective Function.}}
In our primary Picking Variant's experiments, we assumed a penalty factor of $\rho = 10$. To evaluate the structural adaptability of our framework under different operational priorities, we vary this penalty factor across $\rho \in \{0, 1, 10, 25, 50\}$ for both \textit{DDF-VFA} and the benchmark policies. Table~\ref{tab: methods performance results Picking Variant} summarizes the seed-averaged total vehicle travel time (Travel), customer delay (Delay), and the combined objective value (Total).

The results demonstrate a clear operational trade-off: scaling the penalty factor forces policies to aggressively suppress customer delay at the expense of increased routing distance. This behavioral adaptation is uniquely visible at the unpenalized boundary ($\rho = 0$), where the objective collapses exclusively to travel-time minimization. Under this incentive structure, \textit{DDF-VFA} achieves the absolute lowest travel time (2,010 minutes) by allowing customer delays to drift to 1,538 minutes. Notably, this 2,000-minute threshold represents an unavoidable structural routing floor required to fulfill the physical delivery layout, regardless of priority. In contrast, rule-based policies such as \textit{FIFO} lack this systemic optimization awareness; \textit{FIFO} blindly expends high routing overhead (2,279 travel minutes) to maintain a low delay profile (1,317 minutes) that is entirely unpenalized by the underlying objective. This confirms that the learned VFA policy effectively internalizes altered reward architectures rather than relying on invariant heuristics.

Conversely, when the penalty escalates to its highest level ($\rho=50$), \textit{DDF-VFA} successfully minimizes exposure to severe delay weights by shifting its policy to prioritize speed over routing economy. It curtails customer delay to 678 minutes while keeping travel time bounded at 2,149 minutes, yielding a total objective of 36,024. This performance heavily outpaces the \textit{FIFO} and \textit{Integrated} baselines, which yield total values of $68,110$ and $50,287$, respectively. Across all tested parameter settings, the cost reductions achieved by \textit{DDF-VFA} over the complex \textit{AI} benchmark range from 20.2\% to 47.1\%, highlighting the robust scalability of the policy-level decomposition approach under deeply altered operational environments.

\begin{table}[!t]
\caption{Results: Total vehicles travel time (in min.), total delay (in min.), and total objective value for the Picking Variant.}\label{tab: methods performance results Picking Variant}
\centering
\scriptsize
%\tiny
\begin{tabular}{|r|rrr|rrr|rrr|rrr|rrr|}
\hline
\multirow{2}{*}{$\rho$}  &
  \multicolumn{3}{c|}{\textit{DDF-VFA}} &
  \multicolumn{3}{c|}{\textit{FIFO}} &
  \multicolumn{3}{c|}{\textit{Integrated}} &
  \multicolumn{3}{c|}{\textit{AI}} &
  \multicolumn{3}{c|}{\textit{DDF-G}} \\
&  Travel & Delay & Total & Travel & Delay & Total & Travel & Delay & Total & Travel & Delay & Total & Travel & Delay & Total \\
\hline
0 &  %rho%%%%%%%%%%%

\textbf{2,010}  & 1,538   &  \textbf{2,010}& % \textit{DDF-VFA} travel, delay, total
2,279 & \textbf{1,317} & 2,279 & % FIFO Travel, Delay, Total
2,123 & 1,338 & 2,123 & % Integrated Travel, Delay, Total
2,100   &  {1,361} &  2,100 &  % AI Travel, Delay, Total
 2,016     & 1,646  & 2,016   \\ % DDF-G Travel, Delay, Total 
1 &  %rho%%%%%%%%%%%

\textbf{2,048} & \textbf{861}  & \textbf{2,909} & % \textit{DDF-VFA} travel, delay, total
 2,279 & 1,317 & 3,596 & % FIFO Travel, Delay, Total
2,104 &  1,118 & 3,222 & % Integrated Travel, Delay, Total
2,065 &  1,097 &  3,162&  % AI Travel, Delay, Total
2,097 &  1,046 &   3,143  \\ % DDF-G Travel, Delay, Total				
10 &  %rho%%%%%%%%%%%

 2,117&  \textbf{755} & \textbf{9,667} & % \textit{DDF-VFA} travel, delay, total
 2,279 & 1,317 & 15,445 & % FIFO Travel, Delay, Total
 2,115&   1,004& 12,154 & % Integrated Travel, Delay, Total
2,119 &  991 & 12,034 &  % AI Travel, Delay, Total
\textbf{2,111} &  908 &   11,194  \\ % DDF-G Travel, Delay, Total 				
 25 &  %rho%%%%%%%%%%%
		
2,148	   &  \textbf{684}  &  \textbf{19,250}& % \textit{DDF-VFA} travel, delay, total
2,279 & 1,317 & 35,195 & % FIFO Travel, Delay, Total 		
2,143 & 1002 & 27,196 & % Integrated Travel, Delay, Total
 2,189	  & 853 &  23,517  & % AI Travel, Delay, Total
 \textbf{2,120}  &  854 & 23,474  \\ % DDF-G Travel, Delay, Total 
50 &  %rho%%%%%%%%%%%

2,149	   &  \textbf{678}  &  \textbf{36,024}& % \textit{DDF-VFA} travel, delay, total
2,279 & 1,317 & 68,110 & % FIFO Travel, Delay, Total
2,147 & 963 & 50,287 & % Integrated Travel, Delay, Total
 2,203 	  & 858 &  45,122   & % AI Travel, Delay, Total
 \textbf{2,125}  &  865 & 45,392  \\ % DDF-G Travel, Delay, Total 
% 100 &  %rho%%%%%%%%%%%%%%%
% 60 &			
% \textbf{2,124} &  \textbf{677}  &  \textbf{69,804}& % \textit{DDF-VFA} travel, delay, total 
% 2,279 & 1,317 &  133,941
% & % FIFO Travel, Delay, Total
%  2,153& 1,005 & 102,661 & % Integrated Travel, Delay, Total
% 2,199  &  848 &  86,984 & % AI Travel, Delay, Total
% 2,140 & 862  &  88,385  \\% DDF-G Travel, Delay, Total
% 0 &  %rho%%%%%%%%%%%%%%%
% 180 & 
% \textbf{2,010}    &  0 &  \textbf{2,010} &% \textit{DDF-VFA} Travel, Delay, Total
%  2,279 &  0 &  2,279 &  % FIFO Travel, Delay, Total
% 2,126  & 0  & 2,126  &  % Integrated Travel, Delay, Total
% 2,101   & 0  &  2,101  &  % AI Travel, Delay, Total
%    2,017   & 0  & 2,017   \\ % DDF-G Travel, Delay, Total 
\hline
\end{tabular}
\end{table}

\subsubsection{{The Importance of the SPS Performance.}}\label{sec: results for Production Variant}
Since the subproblem ($\mathcal{SP}$) governs sequence-dependent setup constraints, the fidelity of the Subproblem Solver (SPS) directly influences the feasible boundaries of the master policy. If the SPS evaluates scheduling spaces myopically without accounting for sequence transitions, it will frequently flag structurally viable delivery combinations as `Infeasible`. This false-negative behavior artificially restricts the Master Problem's exploration loop. To evaluate this component-level dependency, we compare our proposed anticipatory look-ahead SPS against a simplified, myopic heuristic-based SPS that lacks forward-looking coordination.

While both configurations maintain comparable average delay profiles due to the binding nature of the downstream time windows, the structural operational metrics reveal a massive advantage for the anticipatory architecture. By proactively calculating sequence changeovers, the anticipatory SPS significantly improves vehicle consolidation efficiency, increasing the average number of orders packed per trip from $3.06$ to $3.19$. Simultaneously, it optimizes the preparation sequence to reduce the average changeover setup time from 2.75 down to 2.54 minutes per order. Crucially, the simplified myopic SPS lacks this sequencing agility, causing it to reject dense delivery clusters and force premature, sub-optimal vehicle dispatches. As a result, the anticipatory SPS reduces average violation from $11.22$ to $10.85$ minutes, and the violation rate from $34.39$\% to $33.97$\%. This ablation study confirms that high-fidelity subproblem solution is a prerequisite for achieving global master-policy routing efficiency.

\subsubsection{{Robustness to Data Changes.}}
Setup times play a critical role in determining the complexity of the Production Variant, as they introduce significant constraints on the scheduling process. Unlike the other variants, the variability and heterogeneity in setup times create additional challenges, particularly in optimizing the order preparation sequences. 
We analyze the performance of all methods under different levels of setup time heterogeneity. In our main experiments, the setup times were 0, 5, or 10 minutes. In the first experiment, we increase the setup times of 0 minutes to 3 minutes and decrease the setup times of 10 minutes to 7 minutes. All other setup times, initially set to 5 minutes, remain unchanged. In a second experiment, all setup times are 5 minutes. Thus, all experiments share the same average setup time but differ in the setup time variability.

With a minimal setup time of zero minutes, \textit{DDF-VFA} and \textit{DDF-G} optimize the order preparation sequences to fully exploit the minimal setup times, resulting in low average violations of 10.85 and 11.27 minutes, respectively. When the minimum setup time is increased to 3 minutes, the methods faced significantly higher average violations of 16.36 and 16.44 minutes, respectively. Yet, even when the minimal setup time increase, \textit{DDF-VFA} and \textit{DDF-G} maintain significantly better performance compared to other approaches, indicating their adaptability and effectiveness in handling changes in problem parameters.
In particular, \textit{AI} struggles significantly under these conditions, with its average violations increasing from 31.03 to 77.97 minutes. The reason may be that with less distinct differences in the setup times, the search for effective decisions becomes more nuanced and therefore challenging without decomposition. A similar trend is observed for the case when all setup times are 5 minutes, i.e., all methods experience even higher violations, but in the \textit{DDF-VFA} and \textit{DDF-G}, the increase (to 22.4 and 23.4, respectively) is moderate and maintains better performance compared to the other approaches.
This analysis reveals the critical role of the distribution of setup times, and in particular, the minimal setup times. When the minimal setup times are zero, the policies can sometimes bypass the setups and avoid unnecessary delays. Conversely, even a small increase in setup times (to 3 or 5 minutes) results in more constrained operations and higher costs, underscoring the importance of setup-aware optimization.

\section{Conclusion} \label{sec: conclusion}

In this paper, we introduced the DOFP, a new class of problems consisting of an integration of two interdependent stages in dynamic and stochastic environments. This class encompasses a wide range of fulfillment applications in logistics, healthcare, and food services, where coordinating preparation and delivery stages is critical. Our framework also supports multi-step preparation phases, enhancing its versatility for real-world challenges such as picking followed by packing or machine configuration.
Several promising directions emerge from this work. Instantiating and benchmarking additional variants is a natural next step, including, for example, batched preparation, heterogeneous resources and vehicles, multi-stage preparation, emissions restrictions, and perishability constraints. While our structural policy decomposition conceptually opens the door to these more complex environments, fully implementing the framework in those settings may be non-trivial. 
%For instance, in an on-demand meal environment, the upstream preparation subproblem transforms from a feasibility-check or sequencing routine into a highly complex scheduling puzzle governed by strict perishability limits, rapid order cycles, and tightly synchronized resource constraints. Solving this specialized subproblem well within our iterative master-subproblem loop requires significant computational time, which compromises real-time responsiveness. Developing high-speed, specialized subproblem heuristics capable of sustaining the look-ahead policy under such intense time-dependent constraints represents a compelling and rigorous direction for future research.
A separate, structurally distinct direction is to relax the deterministic-MILP assumption and accommodate \emph{stochastic processing times}, where the per-order processing time \(t^p_i\) is itself random: with \(t^p_i\) stochastic, the resource-availability constraint~\eqref{con:4} and the synchronization constraint~\eqref{con:17} become probabilistic statements, calling for a chance-constrained or sample-average-approximation reformulation that we leave to future work. On the application side, testing additional operational variants, such as batching or multi-stage preparation, could expand the framework’s utility. Future studies could incorporate new constraints such as priority-based processing, vehicle access limits, emissions restrictions, or strict delivery windows, enhancing real-world relevance. Extending the framework to cases where costs are concentrated on the first stage, expanding it to three-stage problems {or incorporating stochastic processing times} presents compelling opportunities. 

% % Acknowledgments here
% \ACKNOWLEDGMENT{
% Gal Neria's research is partially supported by the Israeli Smart Transportation Research Center (ISTRC) and the Council for Higher Education in Israel (VATAT). Michal Tzur’s research is partially supported by the Israeli Smart Transportation Research Center (ISTRC). Marlin Ulmer's work is funded by the DFG Emmy Noether Programme, project 444657906. We gratefully acknowledge their support. 

% \textit{The authors report there are no competing interests to declare.}
% % Enter the text of acknowledgments here
% }% Leave this (end of acknowledgment)

% \section*{End Notes}
% \begin{enumerate}
%     \item See, e.g., the companies \href{https://www.materialsondemand.com}{"Materials on Demand"} and \href{https://www.dispatchit.com/industries/building-products}{"Dispatch"} that operate in the USA.
% \end{enumerate}
% Appendix here
% Options are (1) APPENDIX (with or without general title) or 
%             (2) APPENDICES (if it has more than one unrelated sections)
% Outcomment the appropriate case if necessary
%
\bibliographystyle{elsarticle-harv}
\bibliography{references}
\newpage
\appendix

%\section*{Appendix}

\section{Summary of Notation}\label{app:notation}
Tables \ref{tab:Summary of Notation Part1} and \ref{tab:Summary of Notation Part2} summarize the main notation used in the paper.

% --- TABLE 1: PARAMETERS, INDICES, AND MDP STATE SPACE ---
\begin{table}[h!]
\centering
\scriptsize
\caption{Summary of Notation: Parameters, Indices, and MDP Framework}
\label{tab:Summary of Notation Part1}
\begin{tabular}{|l|l|l|}
\hline
\textbf{Category} & \textbf{Symbol} & \textbf{Description} \\
\hline
General Parameters and Sets &
\( T^c \) & Time horizon for order arrivals \\
& \( T \) & Total time horizon for processing and delivery \\
& \( I \) & Set of all orders \\
& \( C \) & Set of preparation resources (e.g., pickers, machines) \\
& \( V \) & Set of delivery vehicles \\
& \( l_i \) & Customer location of order \( i \) \\
& \( t^o_i \) & Arrival time of order \( i \) to the system \\
& \( t^t_{ij} \) & Travel time between locations \( i \) and \( j \) \\
& \( \kappa \) & Capacity of a delivery vehicle \\
& \( t^p_i \) & Processing time for order \( i \) \\
& \( 0 \) & Location of the central facility (depot) \\
& \( I_k \) & Set of open orders at decision point \( k \) \\
& \( i_k \) & Newly arrived order at decision point \( k \) \\
& \( I_k^{(1)} \) & Information on open orders in preparation stage at decision point \( k \) \\
& \( I_k^{(2)} \) & Information on open orders in delivery stage at decision point \( k \) \\
\hline
Indices &
\( k \) & Decision point index \\
& \( i, j \) & Orders/customers \\
& \( c \) & Preparation resources \\
& \( v \) & Vehicles \\
& \( t \) & Time index \\
& \( \theta \) & Trip index \\
\hline
MDP Notation &
\( S_k \) & State at decision point \( k \), consisting of \( S_k^{(1)} \) (preparation) and \( S_k^{(2)} \) (delivery) \\  
& \( S_k^{(1)} \) & Preparation state at decision point $k$\\
& \( S_k^{(2)} \) & Delivery state at decision point $k$ \\
& \( \mathcal{S} \) & State space \\
& \( S_k^x \) & Post-decision state at decision point \( k \) \\
& \( \mathcal{S}^x \) & Post-decision state space \\
& \( K \) & Total number of decision points (a random variable) \\
& \( \Psi_k \) & Preparation plan at decision point $k$\\
& \( t^s_k \) & Start times for preparation tasks at decision point $k$\\
& \( t^r_{k} \) & Return time of all vehicles from their last departed trips at decision point $k$\\
& \( \Theta_k \) & Delivery plan at decision point $k$\\
& \( \Theta_{kv} \) & Delivery plan for vehicle \( v \) at decision point \( k \) \\
& \( \theta_{kv1}, \theta_{kv2}, \dots \) & Sequence of trips planned for vehicle \( v \) \\
& \( t_k^d \) & Vector of planned departure times at decision point \( k \) \\
& \( x_k \) & Decision at \( k \), comprising \( x_k^{(1)} \) (preparation decision) and \( x_k^{(2)} \) (delivery decision) \\
& \( x_k^{(1)} \) & Preparation decision: \( (\Psi_k^x, t_k^{s,x}) \) \\
& \( x_k^{(2)} \) & Delivery decision: \( (\Theta_k^x, t_k^{d,x}) \) \\
& \( \mathcal{X}(S_k) \) & Action space / set of feasible decisions at state \( S_k \) \\
& \( \mathcal{X}^{(1)}(S_k^{(1)}) \) & Set of feasible preparation decisions at state \( S_k^{(1)} \) \\
& \( \mathcal{X}^{(2)}(S_k^{(2)}) \) & Set of feasible delivery decisions at state \( S_k^{(2)} \) \\
& \( W_{k+1} \) & Stochastic information arriving between decision points \( k \) and \( k+1 \) \\
& \( \Omega \) & Sample space of stochastic information \\
& \( S^M \) & System transition function \\
& \( \pi \) & Combined policy across both stages \\
& \( \pi^{(1)} \) & Policy for preparation decisions in the first stage \\
& \( \pi^{(2)} \) & Policy for delivery decisions in the second stage \\
& \( \Pi \) & Set of all feasible combined policies \\
& \( \Pi^{(1)} \) & Set of feasible preparation policies \\
& \( \Pi^{(2)} \) & Set of feasible delivery policies \\
\hline
\end{tabular}
\end{table}

\clearpage % Forces a page break if needed, or omit if you want them to flow naturally

% --- TABLE 2: DECISIONS, CONSTRAINTS, AND OBJECTIVES ---
\begin{table}[h!]
\centering
\scriptsize
\caption{Summary of Notation: Decisions, Constraints, and Objectives}
\label{tab:Summary of Notation Part2}
\begin{tabular}{|l|l|l|}
\hline
\textbf{Category} & \textbf{Symbol} & \textbf{Description} \\
\hline
Preparation Decision Variables &
\( z_{ict} \) & Indicating if resource \( c \) starts processing order \( i \) at time \( t \) \\
\hline
Delivery Decision Variables &
\( y_{ivt} \) & Indicating if vehicle \( v \) departs with order \( i \) at time \( t \) \\
& \( q_{ijt} \) & Indicating if orders \( i, j \) are delivered sequentially in the same trip that departs at time \( t \)\\
& \( t^a_i \) & Arrival time of order \( i \) at the customer location \\
\hline
Synchronization Constraints &
\( t^m_{ki} \) & End time of preparation for order $i$ that started preparation before \( t_k \)\\
& \( t^d_\theta \) & Departure time of trip \( \theta \) \\
& \( \mathcal{Q}_{prep}(S_k^{(1)}) \) & Set of preparation constraints \\
& \( \mathcal{Q}_{deliv}(S_k^{(2)}) \) & Set of delivery constraints \\
& \( \mathcal{Q}_{sync}(S_k) \) & Set of synchronization constraints \\
\hline
Problem-Specific Parameters
& \( \rho \) & Delay penalty cost per time unit (Picking Variant) \\
& \( \tau \) & Click-to-door promised delivery time (Picking Variant) \\
&\( s_{ij} \) & Setup time between orders \( i \) and \( j \) (Production Variant) \\
& \( t^{sw}_i, t^{ew}_i \) & Soft delivery time window for order \( i \) (Production Variant) \\
& \( d_i \) & Demand size of order \( i \) (Production Variant) \\
\hline
Costs and Objective &
\( \mathbb{C}(S_k^{(2)}, x_k^{(2)}) \) & Immediate cost incurred by second-stage decision \( x_k^{(2)} \) at state \( S_k=(S_k^{(1)},(S_k^{(2)}) \) \\
& \( J(\Theta, t^d) \) & Total cost of a delivery plan $(\Theta, t^d)$ \\
& \( V(S_k) \) & Value function of state \( S_k \) \\
& \( V(S_k^x) \) & Value of the post-decision state of state \( S_k \)\\
\hline
\end{tabular}
\end{table}

\section{Pseudocode for DDF-VFA}\label{app: Pseudocode for DDF-VFA} The general process of \textit{DDF-VFA} is outlined in Algorithm \ref{alg:ddf-vfa}. The framework begins by initializing decisions using a First-In-First-Out (FIFO) strategy, which assigns incoming orders to available resources and vehicle trips as they arrive. The best found solution, denoted $x^*_k$, is initially set to this solution.  
Then, the \textit{DDF-VFA} iteratively explores solutions to the $\mathcal{MP}$ and the $\mathcal{SP}$, as detailed next.

\paragraph{Exploration of solutions to the $\mathcal{MP}$}
An LNS process (described in Section \ref{sec: LNS}) explores potential partial delivery decisions iteratively by modifying the following elements:
(1) Bundles of orders that will be grouped into the same trip;
(2) The sequence in which orders are delivered within each trip; and
(3) The sequence of trip departure times.
For each generated partial delivery decision, $\boldsymbol{\theta}$, the TAS (described 
 in Section \ref{sec: the TAS}) checks feasibility. If TAS declared $\boldsymbol{\theta}$ is infeasible, then the LNS generates another partial delivery decision. Otherwise, TAS assigns specific departure times to each trip, thereby transforming $\boldsymbol{\theta}$ into a complete second-stage decision, $x_k^{(2)}$, which includes all assigned departure times. This step ensures that all trips have departure schedules. 

\paragraph{Obtaining a solution to the $\mathcal{SP}$} Once departure times are assigned, the SPS finds a complement resource assignments, $x_k^{(1)}$, aiming to optimize the cost to go, given $x_k^{(2)}$. That is, together, $x_k^{(1)}$ and $x_k^{(2)}$ form a feasible decision $x_k$, enabling the VFA component to predict the future cost of the associated post decision state, $\hat{V}(S_k^x)$.
The selected $x_k$ is then compared with the best decision found so far, $x^*_k$. If the new decision offers a lower sum of immediate cost and predicted future cost, it replaces the current best solution, i.e., $x^*_k \leftarrow x_k$. The \textit{DDF-VFA} then continues to explore new solutions to the $\mathcal{MP}$.

This iterative process continues until the stopping condition is met, at which point the framework returns the best decision found.

% \SetCommentSty{textnormal} % Ensures comments use normal font style
% %\LinesNumbered % Optional: Enables line numbering for better readability
% \SetAlgoNlRelativeSize{-1} % Adjusts line number font size
% \SetAlgoNlRelativeSize{0}  % Resets default if needed

%\SetAlgoNlRelativeSize{0}
\begin{algorithm}[h]
\caption{\textit{DDF-VFA} Framework}
\label{alg:ddf-vfa}
\begin{algorithmic}[1]
\Require State $S_k$
\Ensure Decision $x_k$

\State Initialize decisions via FIFO assignment
\State Set best decision $x_k^* \gets$ current FIFO solution
\While{stopping criterion not met}
    \State Generate partial delivery decision $\boldsymbol{\theta}$ via LNS  \Comment{Modify trips and order sequence}
    \State Complete $\boldsymbol{\theta}$ to $x_k^{(2)}$ via TAS
    \If{$x_k^{(2)}$ is feasible}
        \State Obtain $x_k^{(1)}$ via SPS given $(S_k, x_k^{(2)})$
        \State $x_k \gets (x_k^{(1)}, x_k^{(2)})$
        \State Predict future cost $\hat{V}(S_k^x)$ via VFA
        \If{$\mathbb{C}(S_k, x_k) + \hat{V}(S_k^x)$ improves best}
            \State $x_k^* \gets x_k$
        \EndIf
    \Else
        \State Continue to next partial decision
    \EndIf
\EndWhile
\State \Return $x_k^*$ \Comment{Best decision found}
\end{algorithmic}
\end{algorithm}

\section{LNS Operators Adjustments}\label{app: LNS Adjustments}
The basic LNS operators used in this paper were borrowed from \citet{NeriaGhost2024}, who addressed a Meal Delivery Variant.
In this appendix, we describe a broader set of operators and adjustments designed for the LNS, {organized by the feature they target. Each variant instantiates the union of operators relevant to its active features: the Picking Variant uses the operators of~\ref{app: LNS Adjustments basic}, and the Production Variant uses those of~\ref{app: LNS Adjustments TW}. 

We note that for all other examples of optional features from \ref{app: MILP extensions} there is no need for new operators. 
However, the feasibility checks might change for: Heterogeneous vehicles, Strict (hard) delivery windows, Vehicle access limits, and Emissions.
None of these adjustments introduce a new operator; each adds a one-shot feasibility test inside the operators. The other optional features that were mentioned do not affect the delivery-side decision space (priority weights, batched preparation, perishability, and multi-stage preparation); and hence, they require no LNS operator changes, as they enter through the cost evaluator and/or the SPS only.}

\subsection{Operators for Travel Time Reduction}\label{app: LNS Adjustments basic}
The Picking Variant aims to minimize both the total travel times and the total delay. While the latter is addressed by the operators designed by \citet{NeriaGhost2024}, we implemented here three new operators that focus on reducing travel times:
\begin{itemize}
    \item \textbf{Swap Operator:} This operator swaps two random orders between two different trips. The objective is to exploit opportunities for better alignment between order locations and their assigned trips, potentially reducing overall travel distance.
    \item \textbf{Insertion Operator:} This operator removes a random order from one trip and reassigns it to another random trip using an insertion heuristic. The heuristic evaluates all possible positions for inserting the order into the target trip and selects the one that minimizes the incremental travel time of the trip. It aims to instantly obtain a better placement to reduce the overall travel time across trips.
    \item \textbf{Reordering Operator:} This operator reorders the sequence of locations within a (random) trip to minimize the total travel time. By solving a traveling salesman problem (TSP) for the trip's locations, it determines the visit sequence that results in the shortest possible route for the trip.
\end{itemize}

\subsection{Operators for Time Windows}\label{app: LNS Adjustments TW}
The Production Variant aims to meet customer time windows, which motivated the addition of the following new operators across all methods:
\begin{enumerate}
\item \textbf{Reordering Three Consecutive Trips by Average Earliest Time Windows}:
Three consecutive trips, $\theta_{(l)}, \theta_{(l+1)}, \theta_{(l+2)} \in \boldsymbol{\theta}$, are randomly selected based on a uniform distribution. These trips are then reordered within the sequence $\boldsymbol{\theta}$ based on the average earliest starting or ending time windows of the orders within these trips (with an equal probability of sorting by starting or ending).
The motivation is reducing time window violations.

   \item \textbf{Sorting Customers within a Trip by Earliest Time Windows}:
  For a randomly selected trip, a random subset of customers are reordered based on their earliest starting or ending time windows (with an equal probability of sorting by starting or ending).
 Here too, the motivation is reducing time window violations.

\item \textbf{Proximity-Based Customer Sorting with Time Windows}:
    For a random trip $\theta \in \boldsymbol{\theta}$, three consecutive customers  $j, j+1, j+2$ selected in random are sorted based on their time window constraints and their proximity to other customers within the same trip. Specifically, for each permutation of the three customers, a  Slack Time is calculated by
    \begin{equation} \text{Slack Time} = \min(t_{[2]}^{ew} - t_{[1]}^{sw} - l_{{[1]}{[2]}}, t_{{[3]}}^{ew} - t_{[2]}^{sw} - l_{{[2]},{[3]}}) \end{equation}
    where $[1], [2], [3]$ refer to the first, second and third customers in the permutation, respectively.
    Then, the selected permutation of these customers is the one maximizing the Slack Time.   
A higher slack time suggests that there is a sufficient buffer for inserting a new order to the trip, reducing the risk of creating future schedule delays or earliness. Conversely, a lower slack time indicates a tighter schedule, which may cause higher delays or earliness in the future.

 \item \textbf{Time Window and Distance Aware Insertion}:
    This operator starts by selecting a random order $i$, which is removed from its current position in its trip and is then inserted into a new position in another random trip $\theta_l \in \boldsymbol{\theta}$ based on the geographical proximity of the sites in $\theta_l$ and the alignment of the trip’s time windows. That is, for each $j \in \theta_l$ we compute a score of inserting order $i$ after it, using the slack time, i.e.,    \begin{equation} \text{Slack Time} = \min\left(t_i^{ew} - t_j^{sw} - l_{ji},  t_{j+1}^{ew} - t_i^{sw} - l_{i,j+1} \right) \end{equation} and we select the location in which the slack time is maximized.    
\end{enumerate}

Two of our benchmark methods, \textit{AI} and \textit{Integrated}, utilize a designated algorithm to assign orders to resources along the LNS process. This algorithm, named PDFT (\textit{Partial Decision Feasibility and Timing}, see \citealt{NeriaGhost2024}), operates based on the sequence of order preparation in a candidate solution. For these benchmark methods, we adapt the PDFT algorithm to account for setup times. Specifically, in the steps where orders are assigned to the first available resource upon completing its last task, the algorithm is modified to assign orders to the first resource available after considering the required setup time.

{
\section{Feature-Specific TAS Extensions}\label{app: TAS feature extensions}
The TAS of Section~\ref{sec: the TAS} accommodates several optional features as small additions to its initial check, departure-assignment, or vehicle-assignment step. Each extension is listed below; the only substantive one, the soft-delivery-window DP, is detailed in~\ref{app: DP for TAS departure times optimization}.
\begin{itemize}
    \item \textbf{Strict (hard) delivery windows:} the postponement step terminates as infeasible if the resulting arrival time would exceed the window upper bound $b_i$ for any order, instead of refining the departure.
    \item \textbf{Heterogeneous vehicles:} the vehicle-assignment step checks the trip's total demand against the vehicle's (specific) capacity ($\kappa_v$).
    \item \textbf{Vehicle access limits:} the TAS draws candidate vehicles for each trip only from the permitted set $\mathcal{A}$.
    \item \textbf{Emissions:} when emissions enter as a hard budget, the TAS rejects a trip-vehicle assignment whose resulting emissions exceed $E_v$; when they enter as an additive objective term, no TAS change is needed.
    \item \textbf{Perishability}: the synchronization-feasibility check (Step~2 of Section~\ref{sec: the TAS}) invokes the perishability-aware SPS variant of~\ref{app: SPS perishability}, which verifies that the chosen departure time admits a preparation completion within $[\hat{t}^a_i - F_i,\, \hat{t}^a_i]$ for every order in the trip.
\end{itemize}
}

\subsection{The TAS Algorithm for Soft Time-Windows}\label{app: DP for TAS departure times optimization}
It is sometimes beneficial to wait between customers' visits to avoid earliness, in contrast to the other two problems where waiting between two customers is not desired.
For this reason, we apply the TAS algorithm in two steps. First, we apply the procedure described in Section \ref{sec: the TAS} to obtain the earliest possible departure times and assignment of trips to vehicles, ensuring that a compatible decision exists for the first stage. 
Then, we adapt the departure times of the trips and set waiting times within the trips. This is done by considering the trips assigned to each vehicle independently.
Thus, the problem that arises obtains as input a partial delivery decision, an assignment of trips to vehicles, their sequence, and the earliest possible departure times. The problem is to schedule the departures from the facility and from each of the visited customers, along the trips, for each vehicle independently. Then, we formulate it as a dynamic program to obtain an optimal solution, as described next.

\textbf{Parameters:} 
Let
${\Theta}_v$ denote the sequence of trips $\theta$ assigned to a certain vehicle $v\in V$ and assign sequential indices to the customers along all the trips in ${\Theta}_v$ according to the trips' sequence, regardless of the trip they belong to. In other words, the index $i$ reflects the overall order in which customers are visited across all trips assigned to the vehicle, e.g., $i=1$ refers to the first customer of the first trip, and so on. 
%$i^{st}_{j}$ and $i^l_{j}$ denote the first ("\textit{st}" is a symbol) and last ("\textit{l}" is a symbol) customer in trip $\theta_j$, respectively;
% $lb_{\theta_j}$ denotes a lower bound on the departure time of trip $\theta_j$. 
Finally, $lb_i$ denotes the earliest possible arrival time at customer $i$, considering the trip departures' lower bounds.

The objective is to minimize the total penalty for earliness and tardiness. The key idea is to determine the optimal arrival time at each customer along its trip while respecting the sequencing constraints within and between the trips.
A DP formulation is given next.

\begin{itemize}
\item\textbf{State Definition.}
The state is defined as
$(i, t)$
where
\( i \) is the current customer
and \( t \in \mathbb{Z}_{\geq {{lb}_i}} \) is the arrival time to this customer (an integer, as travel times and time windows are integers).
\item\textbf{Decision.}
At each state \((i, t) \), the decision is to choose the optimal delivery time \( t_i \in \mathbb{Z}_{\geq t} \) for the current customer from the feasible range \( [t, t_i^{\text{max}}] \), where
$
t_i^{\text{max}} = \max(t_i^{sw}, t)
$. Note that the delivery time $t_i$ may be later than the arrival time, $t$. We also observe that if $t \leq t_i^{sw}$ it is not optimal to set the delivery time to be larger than $t_i^{sw}$ and if $t > t_i^{sw}$ 
it is only optimal to set $t_i=t$ because there is no earliness at customer $i$.
\item\textbf{Cost Calculation.}
At each customer \( i \), the immediate cost (of the DP) \( \mathcal{C}(i, t_i) \) is calculated based on the soft time windows as follows:
$
\mathcal{C}(i, t_i) = \max(0, t_i^{sw} - t_i, t_i - t_i^{ew}).
$
This accounts for the earliness or tardiness penalties associated with the delivery time at customer \( i \).
\item\textbf{{Bellman Equation.}}
The Bellman equation finds the optimal cost $dp(i, t) $ of a state $(i, t)$, i.e., when starting from customer \( i \geq 1\) at time $t \in [0,T]$ and moving forward through the remaining customers:
\[
dp(i, t) = \begin{cases}
\min_{t_i \in [t, t_i^{\text{max}}]} \left( \mathcal{C}(i, t_i) + dp(i+1, t_{i} + tt_{i,i+1}) \right), & i \text{ is not the last site of a trip} \\  
\min_{t_i \in [t, t_i^{\text{max}}]} \left( \mathcal{C}(i, t_i) + dp(i+1,\max(lb_{i+1}, t_{i} +  tt_{i,0}+tt_{0,i+1})) \right),& \text{ otherwise}
\end{cases}
\]
For convenience, we denote by $\mathcal{H}(i, t_i)$ the cost-to-go of decision $t_i$ at state $(i,t)$, i.e.,
\[
\mathcal{H}(i, t_i) = \begin{cases}
 dp(i+1, t_{i} + tt_{i,i+1}), & i \text{ is not the last site of a trip} \\  
  dp(i+1,\max(lb_{i+1}, t_{i} +  tt_{i,0}+tt_{0,i+1})),& \text{ otherwise}
\end{cases}
\]

\item\textbf{Ending Condition.}
\( dp(i,t) = \mathcal{C}(i, t_i^{\text{max}}) \text{ for }i \text{ which is the last customer of the last trip.}\) 
\item\textbf{Starting Condition.}
The objective is given by $dp(1,{lb}_1)$.
\end{itemize}

In the next claim, we introduce a condition that enables us to save computations when solving the DP. To this end, let ${TC}(i, t, t_i)$ denote the sum of the immediate cost and the cost-to-go for selecting the delivery time  $t_i$ at state $(i, t)$. 
\begin{claim}[DP Computation Enhancement]
\label{th:DP Decision Space Reduction} 

\[
\text{If } TC(i, t, t_i + h_1) \geq TC(i, t, t_i), \text{ then, } TC(i, t, t_i + h_1 + h_2) \geq TC(i, t, t_i + h_1 ), \quad \forall i, t, t_i \geq t, h_1 \geq 1,h_2 \geq 1.
\]
% That is, if postponing the delivery time \( t_i \) by 1 minute does not reduce the total cost, no further postponement is required. 
\end{claim}
\begin{proof}
If
\(
{TC}(i, t, t_i + h_1) \geq {TC}(i, t, t_i),
\)
then:
\[
{TC}(i, t, t_i + h_1) - {TC}(i, t, t_i) = {C}(i, t_i + h_1) - {C}(i, t_i) + \mathcal{H}(i, t_i + h_1) - \mathcal{H}(i, t_i) \geq 0.
\]
Note that ${C}(i, t_i + h_1) - {C}(i, t_i) = -h_1, \forall t_i \in [t, t_i^{max}-h_1]$, since any postponement in $t_i$ reduces the earliness at customer $i$.
On the other hand, the cost-to-go $\mathcal{H}(i, t_i)$ is a non-decreasing function of $t_i$. Hence, ${TC}(i, t, t_i + h_1) \geq {TC}(i, t, t_i)$ implies that $\mathcal{H}(i, t_i + h_1) - \mathcal{H}(i, t_i) \geq h_1$, which means that at least one future customer is already late. Further postponing $t_i$ by some $h_2 \geq 1$ will increase the delay for the customers that already experience delays, by $h_2$ for each, and possibly incur delays for other customers as well, while only reducing the immediate cost for $i$ by $h_2$ (for $t_i+h_1+h_2 \leq t_i^{max}$).
\end{proof}
Thus, at a given DP state, if ${TC}(i, t, t_i + h_1) \geq {TC}(i, t, t_i) (t_i \in [t, t_i^{max}-h_1])$, the optimal decision is not greater than $t_i + h_1$ ($h_1 \geq 1$).

{
\section{Feature-Specific SPS Extensions}\label{app: SPS feature extensions}
The SPS of Section~\ref{sec: the SPS} accommodates several optional features. One extension is parameter-style, updating the existing scheduling rule without changing its structure; the rest require a full algorithmic substitution of the SPS. In particular, three of the features listed in \ref{app: MILP extensions} require SPS substitutions which are detailed below; in each case, only the SPS itself changes, and Algorithm~\ref{alg:ddf-vfa} is unaffected. We note that these substitutions may be time consuming and their computational time to performance tradeoff is left for future work.

\subsection{Batched-Preparation SPS}\label{app: SPS batching}
When a single resource can process up to $B$ orders simultaneously, at each step, instead of assigning the next EDD order to a single resource, we form a batch by greedily appending up to $B-1$ additional unassigned orders whose processing fits within the batch finish time $\max_{j \in B_t} t^p_j$ and whose due dates are no later than that finish time. The batch is then assigned as a whole to the resource that minimizes setup time from the resource's last assigned order (or batch), and the resource's availability is advanced by the batch finish time. The procedure repeats until all orders are assigned or no resource can accommodate the next batch by its earliest due date, in which case `infeasible' is returned.

\subsection{Multi-Stage SPS}\label{app: SPS multi-stage}
With $M$ preparation stages, the SPS is reformulated as a flow-shop scheduler. Orders are processed in EDD order; for each order $i$, the algorithm sequentially assigns it to a resource at each stage $m = 1, \dots, M$, subject to the stage-precedence constraint that the start time at stage $m$ for order $i$ is no earlier than its completion at stage $m-1$. The lookahead check of Section~\ref{sec: the SPS} is generalized to verify that at least the next two orders can be feasibly assigned at the \emph{bottleneck stage} (the stage with the largest aggregate workload across resources) without violating their due dates. Feasibility cuts are issued from the bottleneck stage, since infeasibility at any other stage propagates to the bottleneck. The SPS returns stage-$M$ completion times to the MP/SP loop.

\subsection{Perishability-Aware SPS}\label{app: SPS perishability}
Under perishability constraints, the synchronization deadline $b_i = t^d_{\theta(i)} - t^p_i$ on the start of preparation becomes the upper bound of a two-sided window:
\begin{equation}\label{eq: perishability SPS window}
    [a_i', b_i] \;=\; [\,t^d_{\theta(i)} - F_i - t^p_i,\;\; t^d_{\theta(i)} - t^p_i\,].
\end{equation}
The SPS is replaced by a VRPTW-feasibility heuristic acting on these windows, under the correspondence: preparation resources play the role of vehicles, orders play the role of customers, sequence-dependent setups $t_{ij}$ play the role of travel times, processing times are (possibly zero) service times, and routes correspond to per-resource sequences of orders. We suggest using off-the-shelf Google-OR-Tools for the resulting VRPTW problem. 
}

{
\section{MILP Extensions for Optional Features}\label{app: MILP extensions}
The unified MILP of Section~\ref{sec: MILP formulation} is written so that additional features can be switched on by adding a small number of parameters, variables, or constraints. We list some of them below.

\textbf{Heterogeneous resources and vehicles.}\label{sec: ext heterogeneous} Replace the scalar processing time \(t^p_i\) with a resource-indexed processing time \(t^p_{ic}\), and the scalar vehicle capacity \(\kappa\) with a vehicle-indexed capacity \(\kappa_v\). Constraint~\eqref{con:4} becomes \(\sum_{j \in I^{(1)}_k} (t' + t^p_{jc} + s_{ji}) z_{jct'} \leq T + z_{ict}(t - T)\), and constraint~\eqref{con:9} becomes \(\sum_{i \in I^{(2)}_k} d_i \cdot y_{ivt} \leq \kappa_v\). No new variables.

\textbf{Priority-based processing.}\label{sec: ext priority} Assign each order \(i\) a priority weight \(w_i \geq 0\). The objective~\eqref{eq: unified objective} is augmented to weight the earliness/lateness contribution per order by \(w_i\): each bracketed term in~\eqref{eq: unified objective} is multiplied by \(w_{i^\theta_j}\). No new variables or constraints.

\textbf{Strict (hard) delivery windows.}\label{sec: ext strict windows} Add the hard constraint \(a_i \leq \hat{t}^a_i \leq b_i\) for each order \(i\), where \(\hat{t}^a_i\) is the arrival time of \(i\) defined above. In the unified objective~\eqref{eq: unified objective}, set the earliness and lateness penalties to zero (\(\rho_e = \rho_l = 0\)) for orders with hard windows, since infeasibility is now ruled out by the new constraint rather than penalized.

\textbf{Vehicle access limits.}\label{sec: ext vehicle access} Let \(\mathcal{A} \subseteq I_k^{(2)} \times V\) denote the set of permitted order-vehicle pairs (e.g., dictated by vehicle size or driver licenses). Add the constraint \(y_{ivt} = 0\) for all \((i,v) \notin \mathcal{A}\) and all \(t\). No new variables.

\textbf{Emissions restrictions.}\label{sec: ext emissions} Let \(e^t_{ij}\) denote the emissions incurred on the leg from \(i\) to \(j\), and let \(E_v\) denote the per-vehicle emissions budget. Add the constraint \(\sum_{t} \sum_{i,j \in I^{(2)}_k, i \neq j} e^t_{ij} \cdot q_{ijt} \cdot \mathbf{1}\{\text{trip uses vehicle } v\} \leq E_v\) for each \(v \in V\). Implementation uses an auxiliary variable to link \(q_{ijt}\) to the vehicle index; alternatively, the emissions term can be added as a penalty to the objective.

\textbf{Perishability constraints.}\label{sec: ext perishability} Let \(F_i \geq 0\) be the maximum allowed time between completion of preparation for order \(i\) and its arrival at the customer. The completion-of-preparation time for order \(i\) is \(\sum_{c \in C} \sum_t (t + t^p_i) z_{ict}\) (the LHS of synchronization constraint~\eqref{con:17}). Add \(\hat{t}^a_i - \sum_{c \in C} \sum_t (t + t^p_i) z_{ict} \leq F_i\) for all \(i \in I_k\). No new variables.

\textbf{Batched preparation.}\label{sec: ext batching} When a single resource can process up to \(B\) orders simultaneously (e.g., an oven processing multiple meals or a batch picker collecting multiple orders in one walk), constraint~\eqref{con:3} is relaxed to \(\sum_{i \in I_k^{(1)}} z_{ict} \leq B\), and constraint~\eqref{con:4} is generalized so that the resource availability advances by the maximum processing time within the batch rather than the sum: \((\max_{j \in B_t} t^p_j) \cdot \mathbf{1}\{B_t \neq \emptyset\}\), where \(B_t\) denotes the batch active at time \(t\). Linearization of the max introduces auxiliary continuous variables \(u_{ct}\) representing the batch finish time on resource \(c\) starting at \(t\); details depend on the chosen batching policy (fixed-size vs.\ time-window-triggered).

\textbf{Multi-stage preparation.}\label{sec: ext multi-stage} When orders pass through \(M\) preparation stages (e.g., picking followed by packing), introduce a binary variable \(z^{(m)}_{ict}\) for each stage \(m = 1, \dots, M\), each with its own resource set \(C^{(m)}\). Constraints~\eqref{con:2}-\eqref{con:6} are replicated per stage. Stage precedence is enforced by \(\sum_{c \in C^{(m)}} \sum_t (t + t^p_{i,m}) z^{(m)}_{ict} \leq \sum_{c \in C^{(m+1)}} \sum_t t \cdot z^{(m+1)}_{ict}\) for all \(i\) and \(m = 1, \dots, M-1\). The synchronization constraint~\eqref{con:17} is applied with the final stage's completion time.
}

\section{{Additional Results}}\label{app:additional_results}
The aggregated cost figures in the main text mask how each method trades off between the underlying KPIs in the Production Variant. Table~\ref{tab: results with setups and TWs} therefore decomposes performance into average delay, average earliness, average violation, and average setup time across the four Production-Variant instance sets. The original ORT-S and ORT-L are defined in ~\ref{app: instances}; "ORT-S (3-5-5-7 Setup)" and "ORT-S (5-5-5-5 Setup)" are the modified-setup instances from the setup-time sensitivity analysis above.

\begin{table}[t]
\caption{{Total earliness, total delay, total violation, and average setup time (in min.) for the Production Variant.}}
\centering
\scriptsize
\begin{tabular}{|l|lrrrrr|}
\hline
Instances &
  KPI &
  {\textit{DDF-VFA}} &
 {\textit{FIFO}} &
 {\textit{Integrated}} &
 {\textit{AI}} &
{\textit{DDF-G}} \\ \hline
\multirow{4}{*}{ORT-S}  & Avg. Delay (in min.)               &    6.95	&85.35	&26.29&	26.04&	7.36 \\
                        & Avg. Earliness (in min.)                    &   3.89&	4.72&	4.71&	5.00&	3.91 \\
                        & Avg. Violation (in min.)                   &  10.85&	90.07&	30.99&	31.03&	11.27 \\
                        & Avg. Setup Time (in min.)                   &   2.54	&4.88&	3.34	&3.34&	2.56\\
                        \hline
\multirow{4}{*}{ORT-L}  & Avg. Delay (in min.)  &
7.98&	81.71&	19.96	&19.62&	8.26 \\
                        & Avg. Earliness (in min.)                  &    3.81	&5.02	&5.15	&5.17	&3.87  \\
                        & Avg. Violation (in min.)                    &   11.79	&86.74	&25.11	&24.78	&12.12  \\
                        & Avg. Setup Time (in min.)                   &   2.01	&4.86	&3.51	&3.54	&2.08 \\
                        \hline
\multirow{4}{*}{ORT-S (3-5-5-7 Setup)}  & Avg. Delay (in min.)               &                            13.43	&85.01	&76.01	&73.97&	13.21\\
                        & Avg. Earliness (in min.)                    &  2.93&	4.99&	3.87&	4.00&	3.22   \\
                        & Avg. Violation (in min.)                   & 16.36&	90.00&	79.87&	77.97&	16.44   \\
                        & Avg. Setup Time (in min.)                   &  4.01	&4.89&	4.13	&4.19&	4.05  \\
                        \hline
\multirow{4}{*}{ORT-S (5-5-5-5 Setup)}  & Avg. Delay (in min.)               &                          19.61&	83.83&	105.17&	101.83&	20.57\\
                        & Avg. Earliness (in min.)                   & 2.78	&4.69&	3.67	&3.32	&2.83  \\
                        & Avg. Violation (in min.)   &   22.39&	88.52&	108.84	&105.15	&23.40 \\
                        & Avg. Setup Time (in min.)               &                          5.00  &5.00	&5.00	&5.00 &5.00	  \\
                        \hline
\end{tabular}\label{tab: results with setups and TWs}
\end{table}

Two patterns are visible in the table that the aggregate violation numbers above do not show. First, the dominance of \textit{DDF-VFA} on delay (e.g., 6.95 min on ORT-S vs.\ 26.04 for \textit{AI} and 85.35 for \textit{FIFO}) carries over to the larger ORT-L instance (7.98 vs.\ 19.62 and 81.71), confirming that the gain is not specific to instance size. Second, average earliness is consistently small (under 5 min) across all methods, so the violation differences reported earlier are driven almost entirely by the delay side rather than by trading delay against earliness. Setup-time efficiency follows the same ordering: \textit{DDF-VFA} and \textit{DDF-G} produce average setups close to the minimum achievable (2.0-2.5 min on ORT-S/L), while \textit{FIFO}, \textit{Integrated}, and \textit{AI} sit near the unconditional setup-time mean ($\approx$5 min), indicating that the decomposition-based methods successfully exploit the sequence dependency rather than ignoring it.

\subsection{{Results on Independently-Generated Instances}}\label{app: synthetic instances}

In addition to the real-data evaluations of Section~\ref{sec: results objective values all}, which use the Meituan operational dataset for the Picking Variant and the ORTEC-880 instance for the Production Variant, we report results on a second family of instances generated procedurally. We denote them \textbf{MTN-Syn} (Picking-style synthetic instance) and \textbf{PRD-Syn-L} (Production-style synthetic instances, in the large scale). They serve as an additional stress test of the framework on a synthetic geography and arrival process that differ in distribution from the original field datasets, while preserving the structural features (arrival blocks, demand distribution, fleet sizes, setup matrix) used in \ref{app: instances}.

\paragraph{\textbf{MTN-Syn} construction} The synthetic Picking instance retains the paper's per-block Poisson arrival means $(1.3,\,29.5,\,31.1,\,27.7,\,26.7,\,35.9,\,20.5,\,0.5)$, processing time $\sim \mathcal{N}(7,1)$ minutes, click-to-door promise $\tau = 60$ minutes, fleet of five pickers and ten vehicles, and the vehicle speed of 37 distance units per minute used in \ref{app: instances}. It \emph{replaces} the real Meituan customer GPS coordinates - which in the original dataset are clustered around dispatch hotspots - with delivery locations drawn uniformly at random in a $\pm 600$-unit Manhattan-metric box around the depot.

\paragraph{\textbf{PRD-Syn-L} construction} The synthetic Production instances retain the demand distribution $\sim \mathcal{N}(6.2, 2.7)$ clipped to integer support $\{1,\ldots,23\}$, the order-count distributions $\sim \mathcal{N}(300, 22)$, the corresponding fleet size ($38$ vehicles, and $8$ resources), the vehicle capacity of $145$ units, and the setup-time matrix of Table~\ref{tab: setup times}. They \emph{replace} the ORTEC-880 explicit asymmetric distance matrix with uniformly random customer locations in the same $\pm 600$-unit Manhattan box, and \emph{replace} ORTEC's per-minute sampling from a pool of customer time windows with uniform-random arrival times over a 700-minute horizon together with a fixed $[t^o_i + 30,\, t^o_i + 90]$ window for each order.

\paragraph{Comparability caveat} The absolute objective magnitudes reported below are not directly comparable to Tables~\ref{tab: methods performance results Picking Variant} and \ref{tab: results with setups and TWs}. The qualitative ordering of the five policies, however, reproduces in every instance: decomposition methods dominate, with \textit{DDF-VFA} and \textit{DDF-G} an order of magnitude better than \textit{FIFO}, \textit{Integrated}, and \textit{AI}.

\paragraph{Results} Table~\ref{tab: MTN-Syn results} reports the Picking-Variant comparison on MTN-Syn at the same two delay-penalty levels $\rho_l \in \{10, 50\}$ used in Table~\ref{tab: methods performance results Picking Variant}, and Table~\ref{tab: PRD-Syn results} reports the Production-Variant comparison on PRD-Syn-L at $\rho_l = \rho_e = 1$ (matching the setting of the original ORT-L). All numbers are 20-seed averages. \textit{DDF-VFA} dominates on MTN-Syn at both penalty levels and on PRD-Syn-L instance, with the gain over \textit{DDF-G} ranging from $+4.29\%$ to $+11.32\%$.

\begin{table}[ht]
\centering
\scriptsize
\caption{MTN-Syn: 20-seed averages, mean $169.4$ orders per day, 200 training episodes.}\label{tab: MTN-Syn results}
\begin{tabular}{c l rrr c}
\hline
$\rho_l$ & Policy & Total & Routing & Lateness & \textit{DDF-VFA} \%\ red. \\
\hline
\multirow{5}{*}{$10$}
  & FIFO            & $2{,}299{,}491$ & $3{,}883$ & $2{,}295{,}609$ & $+84.60\%$ \\
  & Integrated      & $1{,}659{,}415$ & $3{,}054$ & $1{,}656{,}361$ & $+78.65\%$ \\
  & AI              & $1{,}615{,}900$ & $3{,}048$ & $1{,}612{,}852$ & $+78.08\%$ \\
  & DDF-G           &     $378{,}978$ & $2{,}712$ &     $376{,}267$ &  $+6.53\%$ \\
  & \textit{DDF-VFA} & $\mathbf{354{,}215}$ & $2{,}736$ & $351{,}479$ & - \\
\hline
\multirow{5}{*}{$50$}
  & FIFO             & $11{,}481{,}927$ & $3{,}883$ & $11{,}478{,}044$ & $+84.76\%$ \\
  & Integrated       &  $8{,}207{,}858$ & $3{,}061$ &  $8{,}204{,}798$ & $+78.68\%$ \\
  & AI               &  $8{,}070{,}003$ & $3{,}066$ &  $8{,}066{,}936$ & $+78.31\%$ \\
  & DDF-G            &  $1{,}828{,}513$ & $2{,}779$ &  $1{,}825{,}734$ &  $+4.29\%$ \\
  & \textit{DDF-VFA} & $\mathbf{1{,}750{,}071}$ & $2{,}716$ & $1{,}747{,}355$ & - \\
\hline
\end{tabular}
\end{table}

\begin{table}[ht]
\centering
\scriptsize
\caption{PRD-Syn-L: 20-seed averages, $\rho_l = \rho_e = 1$.}\label{tab: PRD-Syn results}
\begin{tabular}{l l rrrr c}
\hline
Instance & Policy & Total & Routing & Lateness & Earliness & \textit{DDF-VFA} \%\ red. \\
\hline
\multirow{5}{*}{PRD-Syn-L}
  & FIFO             & $1{,}008{,}486$ & $6{,}770$ & $1{,}008{,}482$ &        $4$ & $+88.44\%$ \\
  & Integrated       &     $707{,}229$ & $5{,}180$ &     $707{,}225$ &        $4$ & $+83.51\%$ \\
  & AI               &     $707{,}691$ & $5{,}187$ &     $707{,}687$ &        $4$ & $+83.52\%$ \\
  & DDF-G            &     $131{,}510$ & $4{,}247$ &     $125{,}535$ & $5{,}974$  & $+11.32\%$ \\
  & \textit{DDF-VFA} & $\mathbf{116{,}624}$ & $4{,}330$ & $111{,}897$ & $4{,}727$ & - \\
\hline
\end{tabular}
\end{table}

\section{Instances}\label{app: instances}
In this appendix, we describe in detail how we generated instances for the numerical experiments, based on two different datasets.
\paragraph{\textbf{The Picking Variant}}
We generated instances based on Meituan data, given by the TSL challenge \citep{Meituan24}. We generated an instance with five resources and ten vehicles.
    We assumed the processing times of orders is drawn from a Normal distribution with an average of seven minutes, and standart deviation of one minute.
    The distance between customer locations was calculated using Manhatan distance, and the vehicle average speed was assumed to be 37 distance units to minutes. We selected this speed by observing the driving times to orders that were in solo trips and dividing the Manhatan distance to them by the driving times to them in order to obtain an average speed.
    We assumed the promised delivery time is 60 minutes, which is about the average promised time from the dispatching system in the data plus seven minutes accounted for processing.
We assumed Poisson arrival according to the average number of orders that arrive during five days in two hours intervals (8:00-24:00), in particular, these averages are 
 1.3, 29.5, 31.1, 27.7, 26.7, 35.9, 20.5 and 0.5 orders.
% 08:00 - 10:00: 2.6 orders
% 10:00 - 12:00: 59.0 orders
% 12:00 - 14:00: 62.2 orders
% 14:00 - 16:00: 55.4 orders
% 16:00 - 18:00: 53.4 orders
% 18:00 - 20:00: 71.8 orders
% 20:00 - 22:00: 41.0 orders
% 22:00 - 24:00: 1.0 orders

\paragraph{\textbf{The Production Variant}}
We generated instances based on real data from \citep{ortec2022}. The largest available dataset, containing 880 customer orders, was selected. Below, we describe the data used and additional assumptions made to complete the instance generation. 

The given vehicle capacity is 145 units.
The demand distribution has a mean of 6.2 units (median 6) with a standard deviation of 2.7 units. The maximum demand is 23 units, and the minimum is 1 unit.
In the experiments conducted for the "EURO Meets NeurIPS 2022 Vehicle Routing Competition" (which also used \citet{ortec2022}), dynamic instances were generated by assuming up to 100 orders could arrive at the start of each hour. The actual number of arriving orders followed a uniform distribution between 0 and 100. Customer locations were sampled from the datasets, and for each customer, the time window, demand, and service time were also sampled from the available data. Customers with (strict in their application) time windows that made it infeasible to serve them (e.g., those with windows ending in less than an hour) were excluded.
We adapt the same dynamic instance generation but assume orders arrive on a per-minute basis rather than hourly. Additionally, in our adaptation, we append service times to the travel times, ensuring each route reflects both driving and service duration at customer sites.

We generated both small (ORT-S) and large (ORT-L) instances by varying the number of orders and vehicles, with preparation times linked to order size. This section evaluates the ability of the \textit{DDF-VFA} to balance earliness, delay, and setup times, highlighting its adaptability to both small and large-scale operations.
For ORT-S instances, we considered an average of 200 orders per day, following a normal distribution with a standard deviation of 15 orders, and utilized 25 vehicles. We assumed the availability of 5 resources for the ORT-S instances. For the ORT-L instances, we considered an average of 300 orders per day following a normal distribution with a standard deviation of 22 orders, utilized 38 vehicles, and assumed the availability of 8 resources (i.e., multiplied all by 1.5).
Preparation times, which are not included in the datasets, were determined by the demand size of each order, with the processing time for order $i$ calculated as $t^p_i = d_i$ minutes. Each order was associated with one of four possible classes, and the setup times between orders, which depended on their respective classes, are presented in Table \ref{tab: setup times}. 
% The travel time statistics are as follows: the mean travel time is 33.4 minutes, the median travel time is 31.8 minutes, the maximum travel time is 98.2 minutes, and the standard deviation is 15.8 minutes. 
% %The service times distribution is shown in Figure \ref{fig:demand and service} on the right side.
% The time windows distribution 
% %presented in Figure \ref{fig:time windows},
% includes the following statistics: the mean duration is 2.8 hours, the median duration is 2.1 hours, the minimum duration is 1.9 hours, the maximum duration is 12.5 hours, and the standard deviation of duration is 1.7 hours. Additionally, the mean start time is 219.0 minutes, the mean end time is 387.3 minutes, the minimum start time is 0.0 minutes, the minimum end time is 195.0 minutes, the maximum start time is 380.0 minutes, the maximum end time is 750.0 minutes, the standard deviation of start times is 100.7 minutes, and the standard deviation of end times is 103.8 minutes.

\begin{table}[] 
\caption{The setup times (in minutes) between orders of different classes.}
\centering
\scriptsize
\begin{tabular}{|l|c|c|c|c|}
\hline
\textbf{From /To} & Class 1 & Class 2 & Class 3  & Class 4 \\
\hline
Class 1 & 0 & 5 & 10 & 5 \\
\hline
Class 2 & 5 & 0 & 5 & 10 \\
\hline
Class 3 & 10 & 5 & 0 & 5 \\
\hline
Class 4 & 5 & 10 & 5 & 0 \\
\hline
\end{tabular}\label{tab: setup times}
\end{table}

\end{document}